\documentclass[11pt,reqno]{amsart}
\usepackage[utf8]{inputenc}
\usepackage{a4wide}
\usepackage{appendix}

\usepackage{graphicx}
\usepackage{enumerate}
\usepackage{hyperref}
\usepackage{cleveref}
\usepackage{amsmath,amsopn,amssymb,amsfonts,stmaryrd}
\usepackage{verbatim}
\usepackage{amsthm}
\usepackage{mathtools}
\usepackage{color}
\usepackage{enumitem}
\usepackage[framemethod=TikZ]{mdframed}
\usepackage{bbm}
\usepackage{mathrsfs}
\usepackage{booktabs}
\usepackage{caption}
\usepackage{bm}
\usepackage{tensor}

\makeatletter
\@namedef{subjclassname@2020}{%
	\textup{2020} Mathematics Subject Classification}
\makeatother

\newcommand{\IR}{{\mathbb R}}

\theoremstyle{plain}
\newtheorem{thm}{Theorem}[section]
\newtheorem{cor}[thm]{Corollary}
\newtheorem{lem}[thm]{Lemma}
\newtheorem{prop}[thm]{Proposition}

\newtheorem{rem}[thm]{Remark}

\newtheorem{MainThm}{Theorem}

\newtheorem{Mthm}[MainThm]{Theorem}
\newtheorem*{Mthm*}{Theorem}
\newtheorem{Mcor}[MainThm]{Corollary}
\newtheorem{Mprop}[MainThm]{Proposition}

\theoremstyle{definition}

\newtheorem{defn}[thm]{Definition}

\newtheorem{conj}{Conjecture}

\newtheorem{ques}{Question}
\newtheorem*{weinsteinconj}{Weinstein Conjecture}
\newtheorem*{strongweinsteinconj}{Strong Weinstein Conjecture}

\numberwithin{equation}{section}

\def\a{\alpha}

\def\d{\delta}

\def\k{\kappa}

\def\g{\gamma}

\def\a{\alpha}

\def\d{\delta}

\def\k{\kappa}

\def\g{\gamma}

\def\RR{{\mathbb R}}

\def\d{{\mathrm{d}}}

\def\R{{R_{\alpha}}}

\def\SS{{\mathbb{S}}}

\newcommand{\gc}{\gamma}
\newcommand{\dgc}{\dot{\gc}}

\def\bar{\overline}

\setlist[itemize]{noitemsep, topsep=0pt}

\makeatletter
\newcommand{\vast}{\bBigg@{2}}
\newcommand{\Vast}{\bBigg@{5}}
\makeatother

\newcommand{\RNum}[1]{\uppercase\expandafter{\romannumeral #1\relax}}
\allowdisplaybreaks

\title[The contact type conjecture]{On the contact type conjecture\\ for exact magnetic systems} 
\author{Lina Deschamps}
\address{Faculty of Mathematics and Computer Science,
	University of Heidelberg,
	Im Neuenheimer Field 205,
	69120 Heidelberg, Germany}
\email{ldeschamps@mathi.uni-heidelberg.de}

\author{Levin Maier}
\address{Faculty of Mathematics and Computer Science,
	University of Heidelberg,
	Im Neuenheimer Field 205,
	69120 Heidelberg, Germany}
\email{lmaier@mathi.uni-heidelberg.de}

\author{Tom Stalljohann}
\address{Faculty of Mathematics and Computer Science,
	University of Heidelberg,
	Im Neuenheimer Field 205,
	69120 Heidelberg, Germany}
\email{tstalljohann@mathi.uni-heidelberg.de}



\begin{document}
	\maketitle
\renewcommand{\abstractname}{Abstract}
	\begin{abstract}
In this article, we answer--for a class of magnetic systems--a question now known as the \emph{contact type conjecture}, whose origin trace back to the 1998 work of \emph{Contreras, Iturriaga, Paternain, and Paternain}~\cite{CIPP}. To this end, we explicitly construct, on any \emph{closed manifold}, an \emph{infinite-dimensional space} of exact magnetic systems, which we introduce here as \emph{magnetic systems of strong geodesic type}.

For each such system, there exists at least one \emph{null-homologous embedded periodic orbit} on every \emph{energy level}, with \emph{negative action} for energies below the \emph{strict Mañé critical value}. As a consequence, the corresponding \emph{energy surfaces} are \emph{not of contact type} below this threshold. Thus, for this class of systems, the \emph{contact type conjecture} holds true.

Moreover, for these systems, both the \emph{strict} and the \emph{lowest Mañé critical values} can be computed \emph{explicitly}—and they \emph{coincide} whenever the aforementioned periodic magnetic geodesic is \emph{contractible}, without requiring any additional assumptions on the manifold.

For this class of magnetic systems, several \emph{remarkable multiplicity results} also hold, guaranteeing arbitrarily large numbers of \emph{embedded null-homologous periodic magnetic geodesics} on \emph{every energy level}.

We illustrate the richness of this class of magnetic systems through the following examples. On any \emph{non-aspherical manifold}, there exists a \emph{dense subset} of the space of Riemannian metrics such that, for each such metric, one can construct an \emph{infinite-dimensional space} of exact magnetic fields yielding magnetic systems of \emph{strong geodesic type}. \\
Similarly, on any \emph{closed contact manifold} whose Reeb flow admits at least one null-homologous periodic orbit, one can construct an \emph{infinite-dimensional space} of Riemannian metrics such that, for each such metric, the magnetic system induced by the fixed contact form is of \emph{strong geodesic type}.
	\end{abstract}	
	\section{Introduction}
Following V.~Arnold's pioneering work \cite{ar61} in the 1960s, which placed the motion of a charged particle in a magnetic field into the context of modern dynamical systems, magnetic systems have received considerable attention over the past three decades, particularly with regard to the existence of periodic orbits. To name just a few contributions: \cite{AbbMacMazzPat17,Abbondandolo2015,ABM,Assenza24,AssBenLust16,BM24,CIPP,Fat97, Frauenfelder_Schlenk,Man,M24,maier2025geometrichydrodynamicsinfinitedimensional,Me09,Schlenk_Applications_Hofer_Geom,Sor,Tai83,Tai92a,Tai92b}.

The motion is described mathematically using the language of symplectic geometry ~\cite{Gin} as follows. Let \( (M, g) \) be a closed, connected Riemannian manifold, and let \( \sigma \in \Omega^2(M) \) be a closed two-form, called the \emph{magnetic field}. The triple \( (M, g, \sigma) \) is called a \emph{magnetic system}. The \emph{magnetic geodesic flow} \( \varPhi_{g,\sigma}^t \) is defined on the tangent bundle $TM$ as the Hamiltonian flow determined by the kinetic energy \( E(x,v) = \frac{1}{2} g_x(v, v) \) and the twisted symplectic form
\begin{equation}\label{eq: twisted symplectic form}
    \omega_\sigma := \mathrm{d}\lambda - \pi^*_{TM} \sigma,
\end{equation}
where \( \lambda \) is the metric pullback of the canonical Liouville 1-form from \( T^*M \) to \( TM \) via the metric \( g \), and \( \pi_{TM} \colon TM \to M \) is the canonical projection. The projection of a flow line of the magnetic geodesic flow \(\varPhi_{g,\sigma}^t\) on \(TM\) onto \(M\) is called a \emph{magnetic geodesic}. When \(\sigma = 0\) in~\eqref{eq: twisted symplectic form}, the magnetic geodesic flow coincides with the usual geodesic flow corresponding to the metric \(g\). We note that, unlike standard geodesics, magnetic geodesics cannot, in general, be reparametrized to have unit speed; see \Cref{s: Preliminaries}. 
Therefore, one of the main points of interest is to work out the similarities and differences between standard and magnetic geodesics.

Furthermore, by its Hamiltonian nature, the level sets of the kinetic Hamiltonian are invariant under the magnetic geodesic flow \(\varPhi^t_{g,\sigma}\). More precisely:

For a general magnetic system, the zero section corresponds to the set of rest points of the flow. For all energy levels \( \kappa \in (0, \infty) \), all of which are regular values of the kinetic Hamiltonian \( E \), the corresponding hypersurfaces
\[
\Sigma_{\kappa} := E^{-1}(\kappa),
\]
called the \emph{energy surfaces}, are invariant under the magnetic geodesic flow.
Thus, the dynamics of \(\varPhi^t_{g,\sigma}\) restrict to each energy surface \(\Sigma_\kappa\).

In particular, it is interesting to ask whether the energy surface \( \Sigma_\kappa \) is of contact type, which in case of the geodesic flow is always true. A hypersurface \( \Sigma_{\kappa} \) is said to be of \emph{contact type} if there exists a one-form \( \eta \) on \( \Sigma_{\kappa} \) that is a primitive of \( \omega_\sigma|_{\Sigma_{\kappa}} \), and such that \( \eta \) does not vanish on the line bundle
\[
\ker\left( \left. \omega_{\sigma} \right|_{\Sigma_{\kappa}} \right).
\]
The energy surface \( \Sigma_\kappa \) is of \emph{restricted contact type} if \( \eta \) extends to a one-form on \( TM \). Hypersurfaces of contact type in symplectic manifolds have been widely studied over the last four decades, beginning with the work of Weinstein~\cite{Weinstein78} and Rabinowitz~\cite{Rabinowitz78}, due to their connection with the existence of closed orbits. We provide a brief overview in \Cref{ss: illustr of the main results in light of the Weinstein conjecture}.

\subsection{The contact type conjecture and Mañé's critical values}
\label{subsection: The contact type property, periodic orbits and Mañé's critical value}

This article focuses on determining the range of energy values for which the energy hypersurface \( \Sigma_\kappa \)  is of contact type in the case of exact magnetic systems (i.e. \( \sigma = \mathrm{d}\alpha \) with $\alpha\in \Omega^1(M)$).

To achieve this, the so-called \emph{Mañé critical values}~\cite{Man} play a fundamental role. Among these, the two most significant are the \emph{strict Mañé critical value} \( c_0(M, g, \mathrm{d}\alpha) \in \mathbb{R} \) and the \emph{lowest Mañé critical value} \( c_u(M, g, \mathrm{d}\alpha) \in \mathbb{R} \); see~\eqref{eq: strict mane value} and~\eqref{eq:cu} for definitions. These values serve as energy thresholds that indicate major dynamical and geometric transitions in the magnetic geodesic flow. If \( \alpha \) is not a closed one form, then the following chain of inequalities holds (see, for example,~\cite{Abbo13Lect}):
\begin{equation}\label{eq: inequality of Manes values}
    0 < c_u(M, g, \mathrm{d}\alpha) \leq c_0(M, g, \mathrm{d}\alpha).
\end{equation}
Before proceeding we note that the values \( c_u \) and \( c_0 \) differ in general. This happens, only when the fundamental group of \( M \) is sufficiently non-abelian; see~\cite{FathiMaderna2007}. 

It is well known that for energy levels \( \kappa > c_0(M, g, \mathrm{d}\alpha) \), the corresponding energy surfaces \( \Sigma_\kappa \) are of restricted contact type (see~\cite[Cor.~2]{CIPP}). This naturally raises the question of whether \( \Sigma_\kappa \) is also of contact type for energy levels below the strict Mañé critical value.
According to~\cite[Prop.~B.1]{Co06}, if \( M \) is not a torus, then \( \Sigma_\kappa \) is not of contact type for any \( \kappa \in [c_u(L), c_0(L)) \). 

A long-standing open question--tracing back to the influential work of G. Contreras, R. Iturriaga, G.P. Paternain, and M. Paternain in~\cite{CIPP}, later developed further in~\cite{CFP10,Co06}, and posed as an open problem in 2010 by L. Macarini and G.P. Paternain in \cite[p.2]{MP10}--asks whether the hypersurface \( \Sigma_\kappa \) is of contact type in the low energy range \( 0 < \kappa < c_u(L) \). A. Abbondandolo formulated this question in 2013 explicitly as a conjecture in~\cite[below Thm.~4.2]{Abbo13Lect}:

\begin{conj}[\cite{Abbo13Lect, MP10}]\label{conj: Contact type conjecture}
\( \Sigma_{\kappa} \) is not of contact type for \( 0 < \kappa < c_u(L) \).
\end{conj}

It has been completely resolved in the case where \( M \) is a surface~\cite{ConMacPat2004}, and for certain higher-dimensional homogeneous examples~\cite[Thm.~1.6]{CFP10}. Beyond these cases, as far as the authors know, the \Cref{conj: Contact type conjecture} remains entirely open for general exact magnetic systems in dimension at least three.

In contrast, for symplectic magnetic fields on the two-sphere, there exist examples where \( \Sigma_\kappa \) is of contact type for \emph{all} energy levels (see~\cite{Ben14Contact}).\\ 

We close this subsection by noting that, by~\cite[Prop.~2.4]{ConMacPat2004}, one can affirmatively answer \Cref{conj: Contact type conjecture} if there exists a null-homologous periodic orbit of the magnetic geodesic flow with negative action (for a precise definition, see \Cref{ss: Intermezzo magnetic systems}) for all energy levels below the lowest Mañé critical value. For the sake of completeness, we include this argument in \Cref{Appendix A}.
\subsection{Main results}\label{subsection: Main Results}
The main objective of this article is to confirm \Cref{conj: Contact type conjecture} for a class of magnetic systems. 

To this end, we explicitly construct an infinite-dimensional space of exact magnetic systems for which the conjecture holds; see \Cref{Cor Contact type}.

We begin by describing the setting in which the conjecture is verified. The following two subsections—\Cref{ss: Illustration of main result in light of closed geoddesics} and \Cref{ss: illustr of the main results in light of the Weinstein conjecture}—illustrate the main result through two geometric lenses. Finally, in \Cref{ss: related results}, we discuss how our contributions relate to the existing literature.

To that end, we introduce the following setting:
\begin{defn}\label{Def: magnetic systems of n-geodesic type}
We say that a magnetic system \( (M, g, \sigma) \) is of \emph{geodesic type} if there exists a smooth embedded loop \( \gamma \) in \( M \) that is a geodesic of \( (M, g) \), and such that \( \dot{\gamma} \in \ker \sigma_{\gamma} \), that is,
\[
\dot{\gamma}(t) \in \ker \sigma_{\gamma(t)} := \left\{ v \in T_{\gamma(t)}M \;\middle|\; \sigma_{\gamma(t)}(v, w) = 0 \quad \forall w \in T_{\gamma(t)}M \right\} \quad \forall t \in \mathbb{R} \,.
\]
Such a loop \( \gamma \) is called a \emph{magnetic geodesic of geodesic type} of \( (M, g, \sigma) \).

\end{defn}
\begin{rem}
    In \Cref{Lemm: magnetic geodesic of geodesic type are geodesics and magnetic geodesics}, we see that if \(\gamma\) is a magnetic geodesic of geodesic type of \((M, g, \sigma)\), then it is indeed a geodesic of \((M, g)\) and a magnetic geodesic of \((M, g, \sigma)\).
\end{rem}
We then proceed to illustrate this definition by showing that systems of this type admit at least one embedded periodic orbit for \emph{every} energy level: 

\begin{Mprop}\label{Prop: periodic orbits in all levels}
    Let \( (M, g, \sigma) \) be a magnetic system of geodesic type. Then, for each level of the energy \( \kappa \in (0, \infty) \), there exists at least one periodic orbit of the magnetic geodesic flow \( \varPhi_{g,\sigma}^t \) with prescribed energy \( \kappa \). 
\end{Mprop}
\begin{rem}\label{rem: prop periodic orbits in all levels}
     It follows directly from the proof of \Cref{Prop: periodic orbits in all levels} that, if there exists a contractible magnetic geodesic of geodesic type in $(M,g,\sigma)$ then the magnetic geodesic flow $\varPhi_{g,\sigma}^t$ has at least one contractible embedded periodic orbit at each energy level $\kappa$.
\end{rem}

In the case of an exact magnetic field, \Cref{Def: magnetic systems of n-geodesic type} can be strengthened in order to derive a much stronger statement than \Cref{Prop: periodic orbits in all levels}. For that, we will make the following convention: An embedded loop $\g$ is called \emph{coorientable} if $\g^*TM$ is an orientable bundle over the circle.

\begin{defn}\label{def: exact magnetic systems of stromng geodesic type}
An exact magnetic system \( (M, g, \d\alpha) \) of geodesic type is said to be of \emph{strong geodesic type} if it admits a coorientable, null-homologous, magnetic geodesic $\gc$ of geodesic type of $(M,g,\d\a)$ such that 
\begin{enumerate}
    \item\label{it: 1 strong geodesic defi} the $g$-norm of \(\alpha\) is maximal along \(\gamma\), that is, 
    \[
    \left\vert \alpha_{\gamma(t)} \right\vert_g = \max_{x\in M}\vert \alpha_{x} \vert_g =\Vert \alpha \Vert_{\infty} \qquad \forall t\in \RR \, ,  
    \]
\item\label{it: 2 strong geodeisic defi} the velocity of \(\gamma\) is the metric dual of \(\alpha\) along \(\gamma\) , i.e.,
\[
g_{\gamma(t)}(\dot{\gamma}(t), v) = \alpha_{\gamma(t)}(v) \quad \forall t \in \mathbb{R}, \quad \forall v \in T_{\gamma(t)}M.
\]
\end{enumerate}
Such a loop \(\gamma\) is called \emph{magnetic geodesic of strong geodesic type} of $(M,g,\d\a)$.
\end{defn}

We first comment briefly on this definition. Specifically, we emphasize that an exact magnetic system \((M, g, \mathrm{d}\alpha)\) of strong geodesic type involves a \emph{local} condition on the primitive $\alpha$ of \(\mathrm{d}\alpha\) in a small neighborhood of the loop \(\gamma\), as detailed in~\Cref{def: exact magnetic systems of stromng geodesic type}.

In this setting we have:
 \begin{Mthm}\label{thm: periodic orbits all levels}
    Let \( (M, g, \d\alpha) \) be a magnetic system of strong geodesic type. Then, for every energy level \( \kappa \in (0, \infty) \), the magnetic geodesic flow \( \varPhi_{g,\d\alpha}^t \) admits a null-homologous embedded periodic orbit with energy \( \kappa \), which has non-positive action whenever \( \kappa \in (0, c_0(M, g, \d\alpha)] \). Moreover, the strict Mañé critical value is given by 
    \[
    c_0(M, g, \d\alpha) = \frac{1}{2} \Vert \alpha \Vert_{\infty}^2. \]
   If in addition there exists a magnetic geodesic of strong geodesic type \( \gamma \) in \( (M, g, \d\alpha) \) that is contractible, then
    \[
    c_u(M, g, \d\alpha) = c_0(M, g, \d\alpha) = \frac{1}{2} \Vert \alpha \Vert_{\infty}^2.
    \]
\end{Mthm}  
\begin{rem}
In the proof of \Cref{thm: periodic orbits all levels}, we use all three assumptions—conditions \ref{it: 1 strong geodesic defi} and \ref{it: 2 strong geodeisic defi} from \Cref{def: exact magnetic systems of stromng geodesic type}, as well as the assumption that the periodic magnetic geodesic \( \gamma \) is nullhomologous—to compute the strict Mañé critical value. It is not clear whether the equality \( c_0 = \frac{1}{2} \Vert \alpha \Vert_{\infty}^2 \) still holds if any of these assumptions are dropped.
\end{rem}
This class of magnetic systems exhibits rich dynamics by virtue of \Cref{thm: periodic orbits all levels}, and is significantly broader than one might initially expect, as illustrated by the following theorem and remarks:
\begin{Mthm}\label{thm: B} Let \( M \) be a closed smooth manifold. Then the following two statements hold:\begin{enumerate}
    \item \label{it: thm B 1}  Given an Riemannian metric $g$ on $M$ and a null-homologous, coorientable, simple periodic geodesic \( \gamma \) of $(M,g)$, then one can construct  an exact magnetic field \( \d\a\) such that $\gc$ is a magnetic geodesic of strong geodesic type of \( (M, g, \d\a) \).
    \item \label{it: thm B 2} Given a null-homologous, coorientable, embedded loop \( \gamma \) and a $1$-form $\alpha$ on $M$ so that \(\dot{\gamma} \in \ker(\d\a_{\gamma}) \) and \( \a_\g (\dot \g) \) is constant, then one can construct a Riemannian metric \( g \) such that \( \gamma \) is a magnetic geodesic of strong geodesic type of \( (M, g, \d\a) \).
\end{enumerate}
    In particular, in both cases, the magnetic system \( (M, g, \d\a) \) is of strong geodesic type.
\end{Mthm}
\begin{rem}\label{rem: comment infinite dimensional spaces}
    By the proof of \Cref{thm: B}, for any null-homologous, coorientable, embedded loop \( \gamma\) in  \( M \), one can construct an infinite-dimensional family of magnetic systems \( (M, g, \mathrm{d}\alpha) \) such that the conclusion of \Cref{thm: B} holds. Here and throughout the paper, “infinite-dimensional” is understood in the following sense: The Riemannian metric \( g \) and the magnetic field \( \mathrm{d}\alpha \) are prescribed only in a small neighborhood of the loop \( \gamma \), while outside of such a neighborhood they can be chosen arbitrarily, as long as they satisfy the maximality condition \ref{it: 1 strong geodesic defi} in Definition \ref{def: exact magnetic systems of stromng geodesic type}. In particular, “infinite-dimensional” includes an existence statement.
\end{rem}

    \begin{rem}\label{Rem: constrcution of metrics goes through also for n embededded loops}
An iteration of the construction underlying the proof of \Cref{thm: B} allows to replace everywhere in \Cref{thm: B} the curve \(\gc\) by \(n\) pairwise disjoint curves \(\gc_1, \dots, \gc_n\) having the same properties as \(\gc\).
\end{rem}
\begin{rem}\label{rem: constrcution goes through for magnetic systems of geodesic type}
   The proof of \Cref{thm: B} carries over to magnetic systems of geodesic type \((M, g, \sigma)\), as defined \Cref{Def: magnetic systems of n-geodesic type}, upon replacing “strong geodesic type” with “geodesic type” in the statement. 
\end{rem}
\begin{rem}
The condition that \( \alpha_{\gc}(\dot{\gc}) \) is constant in \ref{it: thm B 2} of \Cref{thm: B} is necessary for constructing an exact magnetic system of strong geodesic type. Indeed~\eqref{eq: energy conservation along gc} implies that, if a magnetic system is of strong geodesic type, this quantity must be constant along \( \gc \).
\end{rem}
Before turning to the main illustration of \Cref{thm: periodic orbits all levels}, we note that \Cref{Rem: constrcution of metrics goes through also for n embededded loops} can be made more precise. Namely, by iterating the construction underlying the proof of \Cref{thm: B} and applying \Cref{thm: periodic orbits all levels}, we obtain the following multiplicity result concerning the existence of embedded periodic orbits on all energy levels.
\begin{Mcor}\label{Cor: multiplicity of magnetic geodesics of strong geodesic type}
 Given \(n\) pairwise disjoint, smooth, coorientable, null-homologous, embedded loops \(\gc_1, \dots, \gc_n\) in \(M\), one can construct an infinite-dimensional space of Riemannian metrics \(g\) and 1-forms \(\a\) on \(M\) such that each loop \(\gc_1, \dots, \gc_n\) is a magnetic geodesic of strong geodesic type for the exact magnetic system \((M, g, \d\a)\). In particular, \((M, g, \d\a)\) is of strong geodesic type.

As a consequence, for every energy level \( \kappa \in (0, \infty) \), the magnetic geodesic flow \( \varPhi_{g,\d\alpha}^t \) admits at least \(n\) null-homologous embedded periodic orbits of energy \( \kappa \), which have non-positive action whenever \( \kappa \in (0, c_0(M, g, \d\alpha)] \).
\end{Mcor}
Experts will recognize that, in light of \Cref{thm: B}, \Cref{conj: Contact type conjecture} follows  directly from \Cref{thm: periodic orbits all levels} together with the inequalities for Ma\~n\'e's critical values~\eqref{eq: inequality of Manes values}. For the reader’s convenience, we detail this argument in \Cref{l: lemma to apply contact type criterion} of \Cref{Appendix A}. It follows that \Cref{conj: Contact type conjecture} holds true for all exact magnetic systems of strong geodesic type. 
\begin{Mcor}\label{Cor Contact type}
Suppose that \( (M, g, \mathrm{d}\alpha) \) is of strong geodesic type. Then the energy level \( \Sigma_{\kappa} \) is not of contact type for any \( \kappa \in (0, c_0] \), and in particular for any \( \kappa \in (0, c_u] \).
\end{Mcor}
In light of \Cref{Cor Contact type}, exact magnetic systems of strong geodesic type provide an appropriate framework to address \Cref{conj: Contact type conjecture}. Accordingly, we now illustrate the richness of this class of systems.

As a first indication, \Cref{Cor: multiplicity of magnetic geodesics of strong geodesic type} implies that for every closed, null-homologous, embedded loop \( \gamma \) in \( M \), there exists — in the sense of \Cref{rem: comment infinite dimensional spaces} — an infinite-dimensional space of exact magnetic systems \( (M, g, \d\alpha) \) in which \( \gamma \) is a magnetic geodesic of strong geodesic type. By \Cref{Cor Contact type}, it follows that for every such loop, there exists a system for which the \Cref{conj: Contact type conjecture} holds.

Moreover, if the geodesic flow of a closed Riemannian manifold $(M, g)$ satisfies the premises of (\ref{it: thm B 1}) in \Cref{thm: B}, then there exists an exact magnetic system of strong geodesic type for which, by \Cref{Cor Contact type}, the \Cref{conj: Contact type conjecture} holds true. This applies, for example, to a \emph{dense set of Riemannian metrics on any non-aspherical manifold}; see \Cref{Cor: non apserical Riemann the contact type cenjecture is true}.

The implication also goes in the other direction. That is, given an exact two-form $\d\a$ with non-trivial kernel, if one can associate to it a vector field whose flow meets the requirements of (\ref{it: thm B 2}) in \Cref{thm: B}, then there exists an exact magnetic system of strong geodesic type for which, by \Cref{Cor Contact type}, \Cref{conj: Contact type conjecture} holds true. For a precise statement, we refer to \Cref{ss: Illustration of it 2 in thm B beyond Reeb dynamics}.

We point out that a natural class of such vector fields is given by Reeb vector fields, which have been intensively studied over the past three decades in connection with the famous Weinstein conjecture. For a detailed discussion, we refer to \Cref{ss: illustr of the main results in light of the Weinstein conjecture}.

We conclude this subsection, and proceed to illustrate the case of rich dynamics of the geodesic flow.
\subsection{Illustration of the main results in light of the Lyusternik-Fet theorem}\label{ss: Illustration of main result in light of closed geoddesics}
   We begin by illustrating the main result in the familiar context of geodesic flow on a closed Riemannian manifold. The existence of closed, nontrivial, and possibly contractible geodesics is a classical topic in Riemannian geometry, dating back to the pioneering work of Hadamard, Poincaré, and Zoll~\cite{hadamard1898,hadamard1899,poincare1890,poincare1905,zoll1903}. In the case of contractible geodesics, this line of research was significantly advanced by Birkhoff~\cite{Birkhoff1917,Birkhoff1966}, culminating in the seminal Lusternik--Fet theorem~\cite{LystFetThm51}, and has since been developed further by many outstanding mathematicians.
   
To state the first result, we introduce some notation. We denote by \( \mathcal{G}^l(M) \) the space of smooth Riemannian metrics on \(M\), endowed with the $C^l$-topology. Recall that the term “infinite-dimensional” used below should be understood in the sense of \Cref{rem: comment infinite dimensional spaces}.
\begin{Mcor}\label{Cor: non apserical Riemann the contact type cenjecture is true}
Let \( M \) be a closed, non-aspherical\footnote{A manifold is called \emph{aspherical} if all its higher homotopy groups vanish, i.e., \(\pi_k(M) = 0\) for all \(k \geq 2\); see~\cite{Luc10} for a survey on this class of manifolds.} manifold of dimension at least two. Then, for each \( 2 \leq l < \infty \), there exists a residual subset of \( \mathcal{G}^l(M) \) such that for every Riemannian metric $g$ therein there exists an infinite-dimensional space of exact magnetic fields \( \mathrm{d}\alpha \) such that \( (M, g, \mathrm{d}\alpha) \) is of strong geodesic type. Consequently:
\begin{enumerate}
    \item \label{it:1 illustr in light of luster} For every energy level \( \kappa \in (0, \infty) \), the magnetic geodesic flow \( \varPhi_{g, \mathrm{d}\alpha}^t \) admits at least one embedded contractible periodic orbit of energy \( \kappa \), which has non-positive action if \( \kappa \in (0, c_0] \).
    \item \label{it:2 illustr in light of luster}The energy level \( \Sigma_{\kappa} \) is not of contact type for \( \kappa \in (0, c_0] \).
     \item\label{it:3 illustr in light of luster}The Mañé critical values of the exact magnetic system \((M, g, \mathrm{d}\alpha)\) are given by
\[
c_u(M, g, \mathrm{d}\alpha) = c_0(M, g, \mathrm{d}\alpha) = \frac{1}{2}\,\Vert \alpha \Vert_{\infty}^2.
\]
\end{enumerate}
\end{Mcor}

\begin{proof}
By the classical theorem of Lyusternik and Fet~\cite{LystFetThm51}, any non-aspherical Riemannian manifold \((M, g)\) admits a non-trivial contractible closed geodesic \(\gamma\). Without loss of generality, we may assume that the geodesic $\gamma$ is prime (i.e, not a multiple cover of another geodesic); otherwise, we replace it with a geometrically equivalent prime geodesic \(\widetilde{\gamma}\), satisfying \(\gamma(\mathbb{S}^1) = \widetilde{\gamma}(\mathbb{S}^1)\). By \cite[Thm.~1]{Rademacherclosedgeodesics2024}, the geodesic $\gamma$ is embedded, since \(g\) is chosen so that it belongs to the generic set considered there. The conclusion follows then from \Cref{thm: periodic orbits all levels}, \Cref{thm: B}, and \Cref{Cor Contact type}.
\end{proof}
On top of that, thanks to the recent result of Contreras-Mazzucchelli~\cite{contreras2024closedgeodesicsbettinumber}, we are able to construct (possibly non-exact) magnetic systems of geodesic type on any closed manifold with non-trivial first Betti number: 
\begin{Mcor}\label{Cor: magentic systems of geodesic type in case non zero betti number}
    Let \( M \) be a closed manifold of dimension at least two with non-trivial first Betti number. Then, for each \( 2 \leq l < \infty \), there exists a dense subset of \( \mathcal{G}^l(M) \) such that for every Riemannian metric $g$ therein there exists an infinite-dimensional space of magnetic fields \( \sigma\) such that \( (M, g, \sigma) \) is of geodesic type and the magnetic geodesic flow $\varPhi_{g,\sigma}^t$ has for each level of the energy $\k\in(0,\infty)$ at least one embedded periodic orbit of energy $\k$. 
\end{Mcor}

\begin{proof}
By \cite[Cor.~B]{contreras2024closedgeodesicsbettinumber}, there exists an open and dense subset of \(\mathcal{G}^l(M)\) such that every Riemannian metric in this set admits (infinitely many) closed geodesics. Choosing one such geodesic, we may assume it is prime, i.e. not a multiple cover of another geodesic. By \cite[Thm.~1]{Rademacherclosedgeodesics2024}, the geodesic is embedded, 
possibly after refining the set of generic metrics by intersecting it with our chosen set (note that the resulting set remains dense as the intersection of an open dense set with a dense set is dense again). \Cref{Prop: periodic orbits in all levels} and \Cref{rem: constrcution goes through for magnetic systems of geodesic type} complete the proof.
\end{proof}
\subsection{Illustration of the main results in the context of the Weinstein conjecture}\label{ss: illustr of the main results in light of the Weinstein conjecture}
As mentioned above, hypersurfaces of contact type in symplectic manifolds have been widely studied due to their connection with the existence of closed orbits. This line of research began with the landmark results of Weinstein~\cite{Weinstein78} and Rabinowitz~\cite{Rabinowitz78} in the late 1970s, and was further developed by Viterbo~\cite{ViterboWeinstein87}, Hofer and Zehnder~\cite{HoferZehnder1987}, and Struwe~\cite{Str90} in the late 1980s.

This study is closely related to the famous Weinstein conjecture, which 
has been central to the development of what is now known as \emph{symplectic dynamics} over the past five decades~\cite{Ginzburg2005}. To state the conjecture precisely, we first introduce the necessary notation.

A \emph{contact manifold} is a pair $(M, \alpha)$, where $M$ is a closed manifold of dimension $2n+1$ and $\alpha$ is a one-form satisfying $\alpha \wedge (\d\alpha)^{n} \neq 0$. Given a contact manifold $(M, \alpha)$, there exists a unique vector field $\R$, called the \emph{Reeb vector field}, defined implicitly by the conditions $\alpha(\R) = 1$ and $\d\alpha(\R, \cdot) = 0$. Its flow, denoted by $\varPhi_{\R}^t$, is called the \emph{Reeb flow} of $(M,\alpha)$. We now state the following:

\begin{weinsteinconj}[\cite{Weinstein78}]\label{conj:Weinsteinconjecture}
Let $(M,\alpha)$ be a closed contact manifold. Then the Reeb flow of $(M,\alpha)$ admits at least one periodic orbit.
\end{weinsteinconj}

The following question can be seen as a strengthening of the Weinstein conjecture.

\begin{ques}\label{ques:strongWeinstein}
    Let $(M,\alpha)$ be a closed contact manifold. Does the Reeb flow of $(M,\alpha)$ admit at least one null-homologous periodic orbit?
\end{ques}

In dimension three, the Weinstein conjecture was resolved by the breakthrough works of Hofer~\cite{Hofer1993} and Taubes~\cite{TaubesWeinstein07}. In higher dimensions, to the best of the authors’ knowledge, it remains only partially understood.

We now state a consequence that holds under a positive answer to \Cref{ques:strongWeinstein}. This result follows directly from \Cref{thm: periodic orbits all levels}, \Cref{thm: B}, and \Cref{Cor Contact type}. We recall that the notion of “infinite-dimensional” should be understood according to \Cref{rem: comment infinite dimensional spaces}.

\begin{Mcor}\label{Cor: Strong Weinstein implies Contytc type conjecture}
    Let $(M,\alpha)$ be a closed contact manifold for which the answer to \Cref{ques:strongWeinstein} is positive. Then there exists an infinite-dimensional space of Riemannian metrics $g$ such that the exact magnetic system $(M,g,\mathrm{d}\alpha)$ is of strong geodesic type. Consequently: \\
    For every energy level \( \kappa \in (0,\infty) \), the magnetic geodesic flow \( \varPhi_{g,\mathrm{d}\alpha}^t \) admits at least one embedded null-homologous periodic orbit of energy \( \kappa \), which has non-positive action if \( \kappa \in (0,c_0] \). Moreover, the conclusions \ref{it:2 illustr in light of luster} and \ref{it:3 illustr in light of luster} of \Cref{Cor: non apserical Riemann the contact type cenjecture is true} hold true.
\end{Mcor}

We now give several classes of contact manifolds for which the answer to \Cref{ques:strongWeinstein} is known to be positive. This list is for illustration and is by no means exhaustive.

In \emph{dimension three}, Hofer proved in \cite[Thm.~1, Thm.~2]{Hofer1993} that the answer to \Cref{ques:strongWeinstein} is positive if \((M,\alpha)\) defines an \emph{overtwisted} contact structure or, in the case of \( (\SS^3,\alpha) \), a \emph{tight} contact structure. For definitions of these terms, see, for example,~\cite{Eliashberg1989, Hofer1993, Gg08}. Moreover, Eliashberg's work~\cite[Thm.~1.6.1]{Eliashberg1989} implies that the space of overtwisted contact structures on an oriented three-manifold~$M$ is homotopy equivalent to the space of plane distributions on~$M$, which highlights the abundance of such structures. As a consequence, the answer to \Cref{ques:strongWeinstein} is known to hold for a broad class of contact three-manifolds.

In \emph{dimension at least five}, the answer to \Cref{ques:strongWeinstein} is known to be positive for several classes of $(M,\alpha)$. For example, this holds when \(\alpha\) defines an \emph{overtwisted} contact structure; see~\cite[Thm.~1]{albers2009weinstein} and the comment following~\cite[Cor.~1.4]{BormanMurphyEliashberg2015}. It is also the case when \((M,\alpha)\) is a compact, simply-connected hypersurface in \(\mathbb{R}^{2n}\); see~\cite{ViterboWeinstein87}. Further instances appear in the contexts of the following works: \cite[Thm.~1.1]{AlbersFuchsMerry2015}, \cite[Thm.~3.1]{GeigesZehmisch2012} and~\cite[Cor.~4]{GeigesZehmisch2016}. 

We now turn to cases where the Weinstein conjecture holds without requiring the stronger version in \Cref{ques:strongWeinstein}. 

In \emph{dimension three}, as previously said, the conjecture was solved in this setting by~\cite{AbbasCieliebakHofer2005, Hofer1993, TaubesWeinstein07}.

In \emph{dimension at least five}, in addition to the examples already covered under the strong version, the Weinstein conjecture holds for the contact manifolds appearing in: \cite[Cor.~3]{AcuMoreno2022}, \cite{FHV89}, \cite[Cor.~1.3]{HV92}.

Combining these results with \Cref{Prop: periodic orbits in all levels} and \Cref{rem: constrcution goes through for magnetic systems of geodesic type}, we obtain the following:

\begin{Mcor}\label{Cor: Illustration of Weinstein}
    Let \((M,\alpha)\) be a closed contact manifold for which the Weinstein conjecture holds (for example, \((M,\alpha)\) belongs to one of the classes above.)\\
    Then there exists an infinite-dimensional space of Riemannian metrics \(g\) such that \((M, g, \mathrm{d}\alpha)\) is of geodesic type, and consequently the conclusion of \Cref{Prop: periodic orbits in all levels} holds.
\end{Mcor} 

\begin{rem}\label{rem: stable Ham Weinstein}
    The Weinstein conjecture and the strong Weinstein conjecture can also be formulated in terms of stable Hamiltonian structures; see~\cite{Abbo13Lect, HoferZehnderBook} for definitions. In dimension three, the Weinstein conjecture was established in this setting by Hutchings-Taubes~\cite{HutchingsTaubes2009}. Moreover, the proof and the conclusion of \Cref{Cor: Illustration of Weinstein} extend to closed manifolds equipped with a stable Hamiltonian structure.
\end{rem}
 We conclude this subsection by highlighting a stronger result related to the Weinstein conjecture: the so called \emph{$n$ or infinity conjecture},  where we refer to~\cite{cineli2024closedorbitsdynamicallyconvex} for an precise discussion of the conjecture. We begin with the discussion of this conjecture in dimension three, where it was formulated by Hofer–Wysocki–Zehnder in~\cite{HoferWysockiZehnder2003}. The recent work~\cite[Thm. 1.1]{cristofarogardiner2024proofhoferwysockizehndersinfinityconjecture} proves this conjecture when the first Chern class of the contact structure induced by \(\alpha\) is zero: the Reeb flow of \((M, \alpha)\) has either two or infinitely many simple periodic orbits. Before stating the next illustration we want to emphasize that by~\cite[Cor.~1.7]{cristofarogardiner2024proofhoferwysockizehndersinfinityconjecture}, a vast class of closed contact manifolds \((M,\alpha)\) of dimension three have infinitely many periodic orbits. 
\begin{Mcor}
\label{Cor: 2 or infinity conjecture} 
    Let \((M,\alpha)\) be a closed contact manifold of dimension three  so that first Chern class of the contact structure induced by \(\alpha\) is zero. Then:
    \begin{enumerate}
        \item There exists an infinite-dimensional space of Riemannian metrics \(g\) such that \((M,g,\mathrm{d}\alpha)\) is of geodesic type, and for each energy level \(\kappa\), the magnetic geodesic flow \(\varPhi^t_{g,\mathrm{d}\alpha}\) has at least two disjoint embedded periodic orbits of energy \(\kappa\).
        \item If in addition \((M,\alpha)\) belongs to the list described in~\cite[Cor.~1.7]{cristofarogardiner2024proofhoferwysockizehndersinfinityconjecture}. Then for each \(n\in\mathbb{N}\), there exists an infinite-dimensional space of Riemannian metrics \(g\) such that \((M,g,\mathrm{d}\alpha)\) is of geodesic type, and for each energy level \(\kappa\), the magnetic geodesic flow \(\varPhi^t_{g,\mathrm{d}\alpha}\) has at least \(n\) disjoint embedded periodic orbits of energy \(\kappa\).
    \end{enumerate}

\end{Mcor}
\begin{proof}
   In the case of two closed Reeb orbits, we can directly conclude the first part of the corollary from \Cref{Rem: constrcution of metrics goes through also for n embededded loops}, \Cref{rem: constrcution goes through for magnetic systems of geodesic type}, and \Cref{Prop: periodic orbits in all levels}. 

Moreover, in the second case, since these orbits are geometrically distinct integral curves of the Reeb vector field, they are pairwise disjoint and we choose $n$ of them. It then follows again by applying \Cref{Rem: constrcution of metrics goes through also for n embededded loops}, \Cref{rem: constrcution goes through for magnetic systems of geodesic type}, and \Cref{Prop: periodic orbits in all levels}.
\end{proof}
In higher dimensions, less is known about the \( n \)- or infinity-conjecture. By the recent work~\cite[Thm.~A]{cineli2024closedorbitsdynamicallyconvex}, we know that if we equip the boundary \( M^{2n-1} \subseteq \mathbb{R}^{2n} \) of a star-shaped domain with its standard contact form \( \alpha \), then the Reeb flow has at least \( n \) simple closed orbits whenever the flow is dynamically convex. For the precise notion of dynamical convexity, we refer, for example, to~\cite{cineli2024closedorbitsdynamicallyconvex} and the references therein. \\
From this result, we can conclude—by an argument following the lines of the proof of \Cref{Cor: 2 or infinity conjecture}—that

\begin{Mcor}\label{Cor: n-infty conjecture}
    Let $M^{2n-1}\subseteq \RR^{2n}$ be the boundary of a star-shaped domain equipped with it's standard contact form $\a$ so that the Reeb flow on $M^{2n-1}$ is dynamically convex. Then there exists an infinite-dimensional space of Riemannian metrics \(g\) such that \((M^{2n-1},g,\mathrm{d}\alpha)\) is of geodesic type, and for each energy level \(\kappa\), the magnetic geodesic flow \(\varPhi^t_{g,\mathrm{d}\alpha}\) has at least $n$ disjoint embedded periodic orbits of energy \(\kappa\).
\end{Mcor}
\begin{rem}
    Since the magnetic systems appearing in \Cref{Cor: magentic systems of geodesic type in case non zero betti number}, \Cref{Cor: Illustration of Weinstein}, \Cref{Cor: 2 or infinity conjecture}, and \Cref{Cor: n-infty conjecture} are exact, the corresponding statements also hold for exact magnetic systems of so-called \emph{semi-strong geodesic type}, as defined later in \Cref{def: semi strong magnetic systems} — provided one replaces “geodesic type” by “semi-strong geodesic type” in the formulation of the corollaries. This yields a genuine strengthening of the results.
\end{rem}
\subsection{Outlook: Illustration of (\ref{it: thm B 2}) in \Cref{thm: B} beyond Reeb dynamics}
\label{ss: Illustration of it 2 in thm B beyond Reeb dynamics}

We emphasize that the application of (\ref{it: thm B 2}) in \Cref{thm: B} is by no means restricted to the case of Reeb vector fields and their null-homologous periodic orbits, as discussed in \Cref{ss: illustr of the main results in light of the Weinstein conjecture}. Rather, this is part of a more general framework, where the proof follows the same strategy as in \Cref{Cor: Strong Weinstein implies Contytc type conjecture}:
\begin{Mcor}
\label{Cor: application of it 2 beyond Reeb}
Let \( M \) be a closed smooth manifold. Given a vector field \( X \in \Gamma(TM) \) and a 1-form \( \alpha \) on \( M \), suppose there exists a coorientable, null-homologous periodic integral curve \( \gamma \) of \( X \) such that
\[
\dot{\gamma}(t) \in \ker (\mathrm{d}\alpha)_{\gamma(t)} \quad \text{and} \quad \alpha_{\gamma(t)}(\dot{\gamma}(t)) = 1 \quad \forall t \in \mathbb{R}.
\]
Then, one can construct an infinite-dimensional space of Riemannian metrics \( g \) such that the exact magnetic system \( (M, g, \mathrm{d}\alpha) \) is of strong geodesic type. Consequently, the three conclusions of \Cref{Cor: non apserical Riemann the contact type cenjecture is true} hold true.
\end{Mcor}
It would be interesting to identify further classes of vector fields, beyond Reeb vector fields, to which \Cref{Cor: application of it 2 beyond Reeb} applies.

\subsection{Related results}\label{ss: related results} This subsection is devoted to explaining how the results presented above fit into the existing literature. Our discussion is structured around three main themes: first, we situate our contributions concerning \Cref{conj: Contact type conjecture}; second, we examine how our multiplicity results for embedded periodic orbits on all energy levels relate to previous work; and lastly, we place our construction of embedded null-homologous periodic orbits—existing on all energy levels and with negative action below Mañé's critical value—in context.
 \medskip

\textbf{The contact type conjecture (see~\Cref{conj: Contact type conjecture})} for exact magnetic systems on closed manifolds of dimension at least three. To the authors' best knowledge, besides the already mentioned homogeneous examples from \cite{CFP10}, published in 2010, nothing else has been known in this direction, and this had remained the last unresolved case. In contrast, \emph{\Cref{Cor Contact type} and its corollaries \Cref{Cor: non apserical Riemann the contact type cenjecture is true}, \Cref{Cor: Strong Weinstein implies Contytc type conjecture}, and \Cref{Cor: application of it 2 beyond Reeb}} settle \emph{\Cref{conj: Contact type conjecture}} for infinite-dimensional spaces of exact magnetic systems on a vast class of manifolds.
\medskip

\textbf{Multiplicity results of embedded magnetic geodesics on all energy levels} are established in \Cref{Cor: multiplicity of magnetic geodesics of strong geodesic type}, \Cref{Cor: 2 or infinity conjecture}, and \Cref{Cor: n-infty conjecture}. These results hold for infinite-dimensional spaces of exact magnetic systems on a large class of closed smooth manifolds. To the authors' best knowledge, \emph{these are the first results of this kind for magnetic systems on closed manifolds of dimension at least three}. Previously, similar results were known only for \emph{almost all} energy levels in magnetic systems on surfaces; see \cite{Abbondandolo2015, AbbMacMazzPat17}. \medskip

\textbf{Constructive nature of the proof.} We emphasize that the proof of \Cref{Prop: periodic orbits in all levels} and \Cref{thm: periodic orbits all levels} is entirely \emph{constructive}: we explicitly construct the periodic orbits, rather than relying on abstract variational methods, such as the minimax principle or Palais--Smale sequences. For an overview of these classical techniques, see~\cite{Abbo13Lect} and the references therein.

This constructive approach is closely related to the construction of the infinite-dimensional spaces of exact magnetic systems of strong geodesic type in \Cref{thm: B}. To the authors' best knowledge, this constitutes the first instance in the literature where such broad classes of magnetic systems have been systematically constructed.
\medskip

\textbf{Embedded null-homologous periodic orbits on all energy levels.} We conclude this subsection by commenting on how \Cref{thm: periodic orbits all levels} fits into the existing literature—more precisely, into the ongoing investigation of periodic orbits of magnetic systems on all energy levels, which has been an intensively studied topic over the past two decades.

We begin with the \emph{main novelty} of \Cref{thm: periodic orbits all levels}. In contrast to~\cite{AssBenLust16, Assenza24, Co06}, we establish, for magnetic systems of strong geodesic type, the existence of a null-homologous \emph{embedded} periodic orbit with \emph{non-positive action} on \emph{every} subcritical energy level $\kappa \leq c_0$. It is unclear to the authors whether this can be derived from any of the results in~\cite{AssBenLust16, Assenza24, Co06}, as the contractible periodic orbits constructed therein neither necessarily have negative action for energies below the lowest Mañé critical value, nor are they necessarily embedded. Additionally, in contrast to~\cite{AssBenLust16}, we do not assume that $M$ is aspherical, and our result holds for \emph{all} energy levels, whereas their result holds only for \emph{almost all} energy levels. In contrast to~\cite{Assenza24}, we do \emph{not} require any curvature assumptions or that the magnetic field be nowhere vanishing.

Additionally, in \Cref{thm: periodic orbits all levels}, we obtain an explicit expression for the strict Mañé critical value in the exact magnetic case of strong geodesic type. Without any restriction on $\pi_1(M)$, we also show that if there exists a magnetic geodesic of strong geodesic type and a contractible loop $\gamma$ in $(M, g, \mathrm{d}\alpha)$, then the lowest and strict Mañé critical values coincide. Previously, this equality was known only under the additional assumption that $\pi_1(M)$ is amenable~\cite{FathiMaderna2007}. For example, \Cref{Cor: multiplicity of magnetic geodesics of strong geodesic type} confirms that our result extends this equality to infinite-dimensional spaces of exact magnetic systems on every manifold with \emph{non-amenable} $\pi_1(M)$.

In addition, as a corollary of \Cref{thm: periodic orbits all levels}, one sees in \Cref{Cor: multiplicity of magnetic geodesics of strong geodesic type}, \Cref{Cor: non apserical Riemann the contact type cenjecture is true}, \Cref{Cor: Strong Weinstein implies Contytc type conjecture}, and \Cref{Cor: application of it 2 beyond Reeb} that the conclusion of \Cref{thm: periodic orbits all levels} holds for infinite-dimensional spaces of exact magnetic systems on a huge class of closed smooth manifolds.

\subsection{Outline of the paper}In \Cref{s: Preliminaries}, we recall the necessary background on magnetic systems and define Mañé's critical values. \\
Then, \Cref{section: Proofs of the main results periodic orbits} is devoted to the proofs of \Cref{Prop: periodic orbits in all levels} and \Cref{thm: periodic orbits all levels}. \\
Finally, we close the paper in \Cref{s: construction of metrics and magnetic fields} with the construction that underlies the proof of \Cref{thm: B}, which is at the heart of the paper. We refer to \Cref{ss: magnetic systems of semi strong geodesic type} for an overview and the key results involved in this construction. As a side product, we also include in \Cref{ss: proof of Cor multiplicity} the proof of \Cref{Cor: multiplicity of magnetic geodesics of strong geodesic type}.
\\
\\
	\noindent \textbf{Acknowledgments:}
The authors are grateful to their advisor, P.~Albers, for his insightful discussions and continuous interest in their work. They further thank all participants of the symplectic research seminar in Heidelberg, as well as M.~Mazzucchelli, G.~Paternain, and F.~Schlenk, for valuable feedback on earlier versions of this article. In addition, they acknowledge A.~Abbondandolo, L.~Assele, V.~Assenza, G.~Benedetti, F.~Ruscelli, and R.~Siefring for helpful discussions.

L.M. gratefully acknowledges L. Asselle and V. Assenza for their warm hospitality during research visits to their home institutions. L.M. also thanks U. Frauenfelder, V. Ginzburg, and B. Gürel for helpful discussions.\\
The authors acknowledge funding by the Deutsche Forschungsgemeinschaft (DFG, German Research Foundation) – 281869850 (RTG 2229), 390900948 (EXC-2181/1) and 281071066 (TRR 191).  L.D. and L.M. would like to acknowledge the excellent working conditions and the stimulating interactions at the Erwin Schrödinger International Institute for Mathematics and Physics in Vienna, during the thematic programme \emph{``Infinite-dimensional Geometry: Theory and Applications”}, where part of this work was carried out.
\section{Preliminaries}\label{s: Preliminaries}
\subsection{Intermezzo magnetic systems}\label{ss: Intermezzo magnetic systems} 
We begin by presenting the mathematical framework used to study the dynamics of a charged particle in the presence of a magnetic field, following V. Arnold's pioneering approach~\cite{ar61}.
    
Let $(M,g)$ be a closed, connected Riemannian manifold and $\sigma\in\Omega^2(M)$ be a closed two-form. The form $\sigma$ is called \emph{magnetic field} and the triple $(M,g,\sigma)$ is called \emph{magnetic system}. This determines the skew-symmetric bundle endomorphism $Y\colon TM\to TM$, the \emph{Lorentz force}, by
\begin{equation}\label{e:Lorentz}
    g_q\left(Y_qu,v\right)=\sigma_q(u,v),\qquad \forall\, q\in M,\ \forall\,u,v\in T_qM.
\end{equation}
We call a smooth curve $\gc\colon \RR\to M$ a \emph{magnetic geodesic} of $(M,g,\sigma)$ if it satisfies \begin{equation}\label{e:mg}
		\nabla_{\dgc}\dgc= Y_{\gc}\dgc
\end{equation}
where $\nabla$ denotes the Levi-Civita connection of the metric $g$. The equation~\eqref{e:mg} reduces to the geodesic equation \(\nabla_{\dot{\gamma}} \dot{\gamma} = 0\) when \(\sigma = 0\), that is, when the magnetic form vanishes. Moreover, \eqref{e:mg} can be viewed as a linear deformation in the velocity $\dgc$ of the geodesic equation.
A central question in the study of magnetic systems is therefore to understand how magnetic geodesics compare with the standard geodesics.\\
Like standard geodesics, magnetic geodesics have constant kinetic energy $E(\gamma,\dot\gamma):=\tfrac12g_\gamma(\dot\gamma,\dot\gamma)$, and hence travel at constant speed $|\dot\gamma|:=\sqrt{g_\gamma(\dot\gamma,\dot\gamma)}$; since the Lorentz force $Y$ is skew-symmetric.
This conservation of energy reflects the Hamiltonian nature of the system, as described at the beginning of the paper. Indeed, the \emph{magnetic geodesic flow} is defined on the tangent bundle by
\[
\varPhi_{g,\sigma}^t\colon TM\to TM,\quad (q,v)\mapsto \left( \gc_{q,v}(t),\dgc_{q,v}(t)\right),\quad \forall t\in\RR,
\]
where $\gc_{q,v}$ is the unique magnetic geodesic with initial condition $(q,v)\in TM$. As shown in~\cite{Gin}, and already mentioned at the beginning of the paper, this flow is Hamiltonian with respect to the kinetic energy $E \colon TM \to \RR$ and the twisted symplectic form
\[
\omega_\sigma = \d\lambda - \pi^*_{TM}\sigma,
\]
where $\lambda$ is the metric pullback of the canonical Liouville $1$-form from $T^*M$ to $TM$ via the metric $g$, and $\pi_{TM} \colon TM \to M$ is the basepoint projection. \\
However, a key difference from standard geodesics is that magnetic geodesics with different speeds are not mere reparametrizations of unit-speed magnetic geodesics. This can be seen, for instance, from the fact that the left-hand side of \eqref{e:mg} scales quadratically with speed, while the right-hand side of \eqref{e:mg} scales only linearly. This makes it natural to study the behavior of the magnetic geodesic flow $\varPhi_{g,\sigma}$ and the geodesic flow of $(M,g)$ at varying energy levels, and to compare it to the geodesic flow of $(M,g)$.\\

In the case of an exact magnetic system $(M,g,\d\a)$, the magnetic geodesic flow admits a Lagrangian formulation, and thus also a variational formulation: The corresponding magnetic Lagrangian is
\begin{align*}
    L \colon TM \to \mathbb{R}, \quad L(q, v) := \tfrac{1}{2} |v|^2 - \alpha_q(v) .
\end{align*}
The magnetic geodesic flow \( \Phi^t_{g,\sigma} \) coincides with the Euler--Lagrange flow \( \Phi_L \) associated with the magnetic Lagrangian $L$, see~\cite{Gin}. That is, a curve \( \gamma \colon [0, T] \to M \) is a magnetic geodesic if and only if it is a critical point of the action functional \( S_L \)
\[
S_L(\gamma) := \int_0^T L(\gamma(t), \dot\gamma(t))\,\mathrm{d}t
\]
among all curves \( \delta \colon [0, T] \to M \) with \( \delta(0) = \gamma(0) \) and \( \delta(T) = \gamma(T) \).

This variational principle prescribes the length \( T \) of the time interval but leaves the energy of \( \gamma \) free. On the other hand, for any given energy level \( \kappa \in \mathbb{R} \), the curve \( \gamma \) is a magnetic geodesic with energy \( \kappa \) if and only if it is a critical point of the action functional \( S_{L+\kappa} \) among all curves \( \delta \colon [0, T'] \to M \) such that \( \delta(0) = \gamma(0) \) and \( \delta(T') = \gamma(T) \), for some arbitrary \( T' > 0 \).

We close this subsection by noting that a magnetic geodesic $\gamma$ of $(M, g, \d\alpha)$ is said to have negative action if $S_{L + \kappa}(\gamma) < 0$, where $L$ is the magnetic Lagrangian of the system~$(M, g, \d\alpha)$ and $\kappa$ is the energy of $\g$.
\subsection{Ma\~n\'e's critical values} 
This variational formulation underlies the definitions of the \emph{Mañé's critical values}, introduced in the seminal works~\cite{CIPP,Man}. These quantities can be interpreted as energy levels marking significant dynamical and geometric transitions in the Euler--Lagrange flow induced by the magnetic Lagrangian \( L \). \\

The \emph{strict Mañé critical value} is
\begin{equation}\label{eq: strict mane value}
    c_0(L) := 
    \inf\left\{ \kappa \in \mathbb{R} \,\middle|\, S_{L+\kappa}(\gamma) \geq 0 \ \forall\, T > 0,\ \forall\, \gamma \in C^\infty(\mathbb{R}/T\mathbb{Z}, M) \text{ homologous to zero} \right\}
\end{equation}
while the \emph{lowest Mañé critical value} is 
\begin{align} \label{eq:cu}
    c_u(L) &:= 
    \inf\left\{ \kappa \in \mathbb{R} \,\middle|\, S_{L+\kappa}(\gamma) \geq 0 \ \forall\, T > 0,\ \forall\, \gamma \in C^\infty(\mathbb{R}/T\mathbb{Z}, M) \text{ contractible} \right\} \, .
\end{align}

We refer to ~\cite{Abbo13Lect} and the references therein for a discussion of the relationships among these critical values, as well as the following chain of inequalities:
\[
0 \leq c_u(L) \leq c_0(L).
\]
For the sake of completeness, we also mention a geometric formulation of the strict Mañé critical value, due to~\cite{CIPP}: it is the smallest energy value containing the graph of a closed one-form on $M$:
\begin{equation}\label{d:mane1}
	c_0(L) = \inf_{\theta} \sup_{q\in M} H(q,\theta_q),
\end{equation}
where the infimum is taken over all closed one-forms $\theta$ on $M$ and $H$ is the magnetic Hamiltonian given by the Legendre dual of $L$, that is
\[
H \colon T^*M \to \mathbb{R}, \quad H(q,p) := \tfrac{1}{2}\lvert p + \alpha_q \rvert^2_q \, .\]
For $\k > c_0(L)$, the level set $\Sigma_k$ encloses the Lagrangian graph $\mathrm{gr}(-\theta)$ and is therefore non-displaceable by Gromov’s theorem~\cite{Gr85}.\\

Finally, we note that the Mañé critical value can also be defined for non-exact magnetic fields, following the works~\cite{CFP10,Me09}, though this generalization lies beyond the scope of this paper. 
\section{Proofs of \Cref{Prop: periodic orbits in all levels} and \Cref{thm: periodic orbits all levels}}\label{section: Proofs of the main results periodic orbits}
This section is devoted to the proofs of \Cref{Prop: periodic orbits in all levels} and \Cref{thm: periodic orbits all levels}. We begin by proving a lemma (see~\Cref{Lemm: magnetic geodesic of geodesic type are geodesics and magnetic geodesics}) essential for establishing \Cref{Prop: periodic orbits in all levels}, which will also be crucial for proving \Cref{thm: periodic orbits all levels}.
\subsection{Setting the stage}We begin this subsection by establishing general facts about magnetic systems \((M, g, \sigma)\) of geodesic type. 
\begin{lem}\label{Lemm: magnetic geodesic of geodesic type are geodesics and magnetic geodesics}
    Let $(M,g,\sigma)$ be a magnetic system of geodesic type and $\gc$ be a magnetic geodesic of geodesic type of $(M,g,\sigma)$. Then the Lorentz force $Y$ of the magnetic system  $(M,g,\sigma)$ vanishes on $\ker\sigma$, i.e., 
    \[
    Y_p(v)=0 \qquad \forall\,(p,v)\in \ker\sigma.
    \]
    In particular, $\gc$ is a geodesic of $(M,g)$ and a magnetic geodesic of $(M,g,\sigma)$.
\end{lem}
\begin{proof}
   Note that by the definition of the Lorentz force in \eqref{e:Lorentz}, we have for all \(p\in M\)
\begin{equation*}
    g_p\bigl(Y_p(v), w\bigr) = \sigma_p(v, w) = 0 \qquad \forall v\in \ker \sigma_p,\ w \in T_p M.
\end{equation*}
Therefore, the Lorentz force \(Y\) of the magnetic system \((M, g, \d\alpha)\) vanishes on the kernel of \(\sigma\), i.e.,
\begin{equation}\label{eq: Lorentz force vanishes along kernel of magnetic field}
    Y_p(v) = 0 \qquad \forall (p,v)\in \ker \sigma.
\end{equation}
As \(\gamma\) is, by~\Cref{Def: magnetic systems of n-geodesic type}, a geodesic of \((M, g)\) and moreover $\dot\g$ lies in the kernel of $\sigma$ by assumption, we can conclude from~\eqref{eq: Lorentz force vanishes along kernel of magnetic field} that
\[
\nabla_{\dot\gamma}\dot\gamma = 0 = Y_{\gamma}\bigl(\dot\gamma\bigr)
\]
which finishes the proof.
\end{proof}
\subsection{Proof of \Cref{Prop: periodic orbits in all levels}}
 Consider \((M, g, \sigma)\) a magnetic system of geodesic type, and let \(\gamma\) be a magnetic geodesic of geodesic type of \((M, g, \sigma)\). By \Cref{Lemm: magnetic geodesic of geodesic type are geodesics and magnetic geodesics}, the embedded loop \(\gamma\) is both a geodesic of \((M, g)\) and a magnetic geodesic of \((M, g, \sigma)\). We conclude that \(\gamma\), along with all its constant-speed reparametrizations, are simultaneously embedded geodesics and magnetic geodesics. This finishes the proof of \Cref{Prop: periodic orbits in all levels}.
\subsection{Preparation for the proof of \Cref{thm: periodic orbits all levels}}

We begin by fixing some notation. Let \( (M, g, \mathrm{d}\alpha) \) be a magnetic system of strong geodesic type, and let \( \gamma \) be a magnetic geodesic of strong geodesic type of \( (M, g, \mathrm{d}\alpha) \). Consider all constant-speed reparametrizations of \( \gamma \), denoted by
\begin{equation}\label{eq: definition gamma_r}
    \gamma_r \colon t \mapsto \gamma(r \cdot t) \quad \text{for } r > 0 \, .
\end{equation}

We close this subsection with the following useful observation, which will be helpful in the proof that follows. By \Cref{def: exact magnetic systems of stromng geodesic type}, the $g$-norm of \( \alpha \) is maximal along $\gc$, thus also along every constant-speed reparametrization \( \gamma_r \) of \( \gamma \); that is,
\begin{equation}\label{eq: norm along gcr max}
    \left| \alpha_{\gamma_r(t)} \right|_g = \| \alpha \|_{\infty} \quad \forall\, t \in \mathbb{R} \, .
\end{equation}

\subsection{Proof of \Cref{thm: periodic orbits all levels}} 
Let $(M, g, \d\alpha)$ be a magnetic system of strong geodesic type. We begin by proving that the strict Mañé critical value $c_0(M, g, \d\alpha)$ is bounded from above by $\frac{1}{2}\|\alpha\|_{\infty}^2$. Indeed, by the definition of the strict Mañé critical value in~\eqref{eq: strict mane value}, it is sufficient to prove that the Lagrangian $L + \k$ for $\k := \frac{1}{2} \| \a\|_{\infty}^2$ is nonnegative. This is the case because for all $(p,v)\in TM$ it holds that:
\[
L(p,v) + \k = \frac{1}{2}g_p(v,v) - \alpha_p(v) + \k \ge \frac{1}{2}|v|_g^2 - |v|_g \|\alpha\|_{\infty} + \frac{1}{2}\|\alpha\|_{\infty}^2 = \frac{1}{2}\bigl(|v|_g - \|\alpha\|_{\infty}\bigr)^2 \ge 0.
\]
In order to prove that the previously obtained upper bound on the Mañé critical value is actually an equality, i.e.,
\begin{equation}\label{eq: vert alpha vert = strict value}
    c_0(M, g, \d\alpha) = \frac{1}{2}\|\alpha\|_{\infty}^2,
\end{equation}
it suffices, by definition~\eqref{eq: strict mane value}, to show that for all \(\kappa < \frac{1}{2}\|\alpha\|_{\infty}^2\), there exists a null-homologous magnetic geodesic of \((M, g, \d\alpha)\) with prescribed energy \(\kappa\) and negative action.

To this end, let \(\gamma\) be a magnetic geodesic of strong geodesic type of \((M, g, \d\alpha)\) with prescribed energy \(\kappa_{\gc} \in (0, \infty)\). Thus, it has speed \(\sqrt{2\kappa_{\gc}}\). Since \(\gamma\) is of strong geodesic type, by item~\ref{it: 1 strong geodesic defi} in \Cref{def: exact magnetic systems of stromng geodesic type}, the norm of \(\alpha\) is maximal along \(\gamma\), i.e.,
\[
|\alpha_{\gamma}|_g = \|\alpha\|_{\infty}.
\]
Moreover, by definition (see~(\ref{it: 2 strong geodeisic defi}) in \Cref{def: exact magnetic systems of stromng geodesic type}), the velocity \(\dot\gamma\) is the metric dual of \(\alpha\) along \(\gamma\), which implies 
\begin{equation}\label{eq: vert alpha along gc}
    \|\alpha\|_{\infty}^2 = |\a_\g |_g^2 = |\dot \gamma|_g^2 = g_\g(\dot \g , \dot \g) = \a_\g(\dot \g) = 2 \kappa_\g \, .
\end{equation}
So in particular \(\kappa_{\gc} = \frac{1}{2}\,\|\alpha\|_{\infty}^2\). Furthermore, by an argument following the lines of the proof of \Cref{Prop: periodic orbits in all levels}, the constant-speed reparametrization \(\gamma_r\) defined as in~\eqref{eq: definition gamma_r} is an embedded periodic magnetic geodesic of \((M, g, \d\alpha)\). Combining this with the expression for speed and energy for $\gc$ from~\eqref{eq: vert alpha along gc}, shows that
\begin{equation}\label{eq: energy of gc_r and speed of gc_r}
    |\dot{\gamma}_r|_g = r \cdot \|\alpha\|_{\infty}
    \qquad \text{and}\qquad
    E(\gamma_r, \dot{\gamma}_r) = \frac{1}{2}\vert \dgc_r\vert_g^2=\frac{r^2}{2}\,\|\alpha\|_{\infty}^2\, .
\end{equation}
For the convenience of the reader, we recall that these quantities are mutually dependent integrals of motion, since the energy itself is an integral of motion.
The magnetic Lagrangian
\[
L + \frac{r^2}{2}\,\|\alpha\|_{\infty}^2
\]
evaluated along \(\bigl(\gamma_r, \dot{\gamma}_r\bigr)\) reads as
\begin{equation}\label{eq: langrang evaluated along gc_r}
    L(\gamma_r, \dot{\gamma}_r) + \frac{r^2}{2}\,\|\alpha\|_{\infty}^2 = \frac{1}{2}\, g_{\gamma_r}(\dot{\gamma}_r, \dot{\gamma}_r)
    - \alpha_{\gamma_r}(\dot{\gamma}_r) + \frac{r^2}{2}\,\|\alpha\|_{\infty}^2.
\end{equation}
By \eqref{eq: norm along gcr max}, together with~\eqref{eq: vert alpha along gc} and~\eqref{eq: energy of gc_r and speed of gc_r}, the expression for the magnetic Lagrangian evaluated along $(\gc_r,\dgc_r)$ in~\eqref{eq: langrang evaluated along gc_r} becomes
\begin{equation}\label{eq: magnetic Langrangian along gc_r}
     L(\gamma_r, \dot{\gamma}_r) + \frac{r^2}{2}\,\|\alpha\|_{\infty}^2 = r^2\,\|\alpha\|_{\infty}^2
     - r\,\|\alpha\|_{\infty}^2= r\,(r-1)\,\|\alpha\|_{\infty}^2\, .
\end{equation}
As \(r \neq 0\) and \(\|\alpha\|_{\infty} \neq 0\), this value is negative whenever \(r < 1\). 
    So in summary, we have proven that for all \(\kappa < \frac{1}{2}\,\|\alpha\|_{\infty}^2\), there exists a null-homologous periodic magnetic geodesic of \((M, g, \d\alpha)\) of energy \(\kappa\) and negative action, namely the magnetic geodesic $\g_r$  of $(M,g,\d\a)$ for $0 < r < 1$ 
        given by $\k= \frac{r^2}{2} \|\a\|^2_{\infty} \,$.
This also proves the equality~\eqref{eq: vert alpha vert = strict value}.\\
We finish this proof by noting that if $\gamma$ is contractible instead of merely null-homologous, then by an argument following precisely the lines of the proof up to this point and using the definition of the lowest Mañé critical value~\eqref{eq:cu}, we can conclude that
\[
c_u(M, g, \d\alpha) = \frac{1}{2}\|\alpha\|_{\infty}^2,
\]
which, together with the previously proven equality in~\eqref{eq: vert alpha vert = strict value}, completes the proof.

\section{Proofs of \Cref{thm: B} and \Cref{Cor: multiplicity of magnetic geodesics of strong geodesic type}} \label{s: construction of metrics and magnetic fields}
This section is devoted to the proof of \Cref{thm: B}, and to showing how the multiplicity result in \Cref{Cor: multiplicity of magnetic geodesics of strong geodesic type} can be derived from it. To that end, we introduce in \Cref{ss: magnetic systems of semi strong geodesic type} a weakened version of exact magnetic systems of strong geodesic type, called magnetic systems of \emph{semi-strong geodesic type}, defined in \Cref{def: semi strong magnetic systems}. We then show that a weakened version of \Cref{thm: B} holds in this setting (\Cref{prop: existence of semi strong pair}), along with a rescaling property of these systems (\Cref{prop: rescaling of magnetic systems of semi strong geodesic type}) that allows us, under a mild assumption, to rescale a magnetic system of semi-strong geodesic type into one of strong geodesic type.

In \Cref{ss: proof of thm b}, we derive \Cref{thm: B} from the previously mentioned \Cref{prop: existence of semi strong pair} and \Cref{prop: rescaling of magnetic systems of semi strong geodesic type}, under the aforementioned mild assumption.

\Cref{ss: constrcution of smei strong systems} and \Cref{ss: rescaling} are devoted to the proofs of \Cref{prop: existence of semi strong pair} and \Cref{prop: rescaling of magnetic systems of semi strong geodesic type}, respectively.

We then conclude the paper in \Cref{ss: proof of Cor multiplicity} by deriving \Cref{Cor: multiplicity of magnetic geodesics of strong geodesic type} as a consequence of the proofs of \Cref{prop: existence of semi strong pair} and \Cref{prop: rescaling of magnetic systems of semi strong geodesic type}.
\subsection{Exact magnetic systems of semi-strong geodesic type}\label{ss: magnetic systems of semi strong geodesic type}
In order to prove \Cref{thm: B}, we introduce the following definition, which refines \Cref{Def: magnetic systems of n-geodesic type} and weakens \Cref{def: exact magnetic systems of stromng geodesic type}. At the end of this subsection, we comment more precisely on how these three definitions are related.
\begin{defn}
\label{def: semi strong magnetic systems}
An exact magnetic system \( (M, g, \d\alpha) \) of geodesic type is said to be of \emph{semi-strong geodesic type} if it admits a coorientable magnetic geodesic of geodesic type $\gc$ of $(M,g,\d\a)$ such that  the velocity of \(\gamma\) is the metric dual of \(\alpha\) along \(\gamma\) , i.e.,
\[
g_{\gamma(t)}(\dot{\gamma}(t), v) = \alpha_{\gamma(t)}(v) \quad \forall t \in \mathbb{R}, \quad \forall v \in T_{\gamma(t)}M.
\]
Such a loop \(\gamma\) is called \emph{magnetic geodesic of semi-strong geodesic type} of $(M,g,\d\a)$.
\end{defn}
We begin by establishing that, on any given smooth closed manifold, the space of exact magnetic systems of semi-strong geodesic type is fairly large. This is displayed by explicitly constructing such systems in the following result:
\begin{prop}
\label{prop: existence of semi strong pair}
Let $\gamma$ be an embedded, coorientable smooth loop $\g$ in $M$. Then the following hold: 
\begin{enumerate}[label=(\arabic*)]
    \item \label{it: 1 prop existnece of semi strong pair}For a given a Riemannian metric $g$ on $M$ so that $\g$ is a geodesic of $(M, g) \,$, one can construct an infinite dimensional space of $1$-forms $\a$ on $M$ so that $\g$ is a magnetic geodesic of semi-strong geodesic type of $(M,g , \d \a) \,$.
    \item \label{it: 2 prop existnece of semi strong pair}For a given a $1$-form $\a$ on $M$ so that 
    \[
    \dot \g(t) \in \ker \d \a_{\g(t) } \qquad\text{and} \qquad \a_{\g(t)}(\dot \g(t)) = \mathrm{const.} > 0 \qquad \forall t\in \RR
    \]
  one can construct an infinite dimensional space of Riemannian metrics $g$ on $M$ so that $\g$ is a magnetic geodesic of semi-strong geodesic type of $(M, g , \d \a) \,$.
\end{enumerate}
In particular, in both cases, the magnetic system \( (M, g, \d\a) \) is of semi-strong geodesic type.
\end{prop}
The proof of \Cref{prop: existence of semi strong pair} will be given in \Cref{ss: constrcution of smei strong systems}. We begin by observing that for a given loop $\g$ to be a magnetic geodesic of semi-strong geodesic type of $(M,g , \d\a)$ imposes only local conditions on $g$ and $\a$ near $\g \,$, that is altering $g$ respectively $\a$ outside a neighborhood of $\g$ will not change this property. Moreover, following line by line the computation in~\eqref{eq: vert alpha along gc}, the following identity holds for any magnetic geodesic $\gc$ of semi-strong geodesic type of $(M, g, \d\a)$:
\begin{equation}\label{eq: energy conservation along gc}
    |\a_\g|_g^2 = |\dot \g|_g^2 = \a_\g(\dot \g) 
\end{equation}
and this expression is constant in $t$ since $\g$ is a geodesic of $(M,g) $.
This naturally raises the question of when the quantity on the left-hand side of~\eqref{eq: energy conservation along gc} attains its maximum along $\gc$. As we will see, this is closely related to a key structural property of the space of semi-strong exact magnetic systems. 

Somewhat surprisingly, this space on a given manifold is more flexible than one might initially expect. In particular, it is closed under certain rescalings of the Riemannian metric and the magnetic field. This is made precise in the following result:
\begin{prop}\label{prop: rescaling of magnetic systems of semi strong geodesic type}\
Given an exact magnetic system $(M, g, \d\a)$ of semi-strong geodesic type and a magnetic geodesic $\gc$ of semi-strong geodesic type of $(M, g, \d\a)$, one can construct smooth strictly positive functions $\varrho_1  , \, \varrho_2$ on $M$ such that the following holds:
\begin{enumerate}
    \item \label{it 1: rescaling} For $\tilde g:= \varrho_1 \cdot g$ the curve $\gc$ is a magnetic geodesic of semi-strong geodesic type of $(M,\tilde g,\d\a )$ so that the $\tilde g$-norm  of $\alpha$ is maximal along $\gc$, that is
    \[
    \vert \a_{\gc(t)}\vert_{\tilde g}=\max_{x\in M}\vert \a_x \vert_{\tilde g}=: \Vert \alpha \Vert_{\infty} \qquad \forall t\in \RR\,.
    \]
    \item \label{it 2: rescaling}For $\tilde \a:= \varrho_2 \cdot \a$ the curve $\gc$ is a magnetic geodesic of semi-strong geodesic type of $(M, g,\d\tilde\a )$ so that the $g$-norm of  $\Tilde\a$ is maximal along $\gc$, that is
    \[
    \vert \tilde\a_{\gc(t)}\vert_{g}= \max_{x\in M}\vert \tilde\a_x \vert_g=:\Vert \tilde\alpha \Vert_{\infty} \qquad \forall t\in \RR\,.
    \]
\end{enumerate}
In particular, in both cases, the magnetic systems \( (M,\tilde g, \d\a) \) and \( (M, g, \d\tilde\a) \) are of semi- strong geodesic type.
\end{prop}
\begin{rem}
    For readability, we slightly abuse notation by omitting the dependence of the maximum norms on the chosen Riemannian metric in \ref{it 1: rescaling} and \ref{it 2: rescaling} of \Cref{prop: rescaling of magnetic systems of semi strong geodesic type}.
\end{rem}
\Cref{prop: rescaling of magnetic systems of semi strong geodesic type} exhibits, in particular, the distinction between exact magnetic systems of strong and semi-strong geodesic type. Specifically, given an exact magnetic system $(M, g, \d\a)$ of semi-strong geodesic type, after a suitable rescaling of either $g$ or $\a$, the resulting magnetic systems $(M, \tilde{g}, \a)$ and $(M, g, \d\a)$, as constructed in \Cref{prop: rescaling of magnetic systems of semi strong geodesic type}, satisfy items \ref{it: 1 strong geodesic defi} and \ref{it: 2 strong geodeisic defi} in \Cref{def: exact magnetic systems of stromng geodesic type}. 

However, this alone does not imply that these systems are of strong geodesic type, since it would require the magnetic geodesic $\gc$ to be null-homologous in $M$. By adding this assumption, we can conclude the following:
\begin{cor}\label{Cor: from semi strong to strong}
   Let the setting be as in \Cref{prop: rescaling of magnetic systems of semi strong geodesic type}. If, in addition, the magnetic geodesic $\gc$ of semi-strong geodesic type is null-homologous in $M$, then $\gc$ is a magnetic geodesic of strong geodesic type of the systems $(M, \tilde{g},\d \a)$ and $(M, g, \d\tilde\a)$ constructed in \Cref{prop: rescaling of magnetic systems of semi strong geodesic type}.\\
  In particular, both magnetic systems \( (M, \tilde{g}, \d\a) \) and \( (M, g, \d\tilde{\a}) \) are of strong geodesic type.
\end{cor}
We close this subsection by using \Cref{Cor: from semi strong to strong} to clearly distinguish between magnetic systems of semi-strong geodesic type and those of strong geodesic type. As a starting point, we mention that \Cref{Def: magnetic systems of n-geodesic type}, \Cref{def: exact magnetic systems of stromng geodesic type}, and \Cref{def: semi strong magnetic systems} yield the following natural hierarchy of sets of classes of magnetic systems on a closed smooth manifold:
\begin{equation}\label{eq: hierachy of magnetic systems}
\{\text{strong geodesic type}\} \subseteq \{\text{semi-strong geodesic type}\} \subseteq \{\text{geodesic type}\}\, .
\end{equation}
We note that \Cref{Cor: from semi strong to strong} guarantees that the first inclusion in \eqref{eq: hierachy of magnetic systems} is strict. The second inclusion is strict as well, since many non-exact magnetic systems of geodesic type fail to be semi-strong due to the magnetic field not being exact. See, for instance, \Cref{Cor: magentic systems of geodesic type in case non zero betti number} and \Cref{rem: stable Ham Weinstein}.
\subsection{Proof of \Cref{thm: B}}\label{ss: proof of thm b} 
The statement of \Cref{thm: B} follows directly from \Cref{Cor: from semi strong to strong}, where we assume that the curve $\gc$ in \Cref{prop: existence of semi strong pair} and \Cref{prop: rescaling of magnetic systems of semi strong geodesic type} is null-homologous in $M$.
\subsection{The construction of exact magnetic systems of semi-strong geodesic type}\label{ss: constrcution of smei strong systems}
This subsection is devoted to the proof of \Cref{prop: existence of semi strong pair}. As the conditions in \Cref{def: semi strong magnetic systems} are local in nature, it suffices to construct the $1$-form~$\alpha$ (in~\ref{it: 1 prop existnece of semi strong pair}) and the Riemannian metric~$g$ (in~\ref{it: 2 prop existnece of semi strong pair}) in a neighborhood of the loop~$\gc$. The precise meaning of this localization will be clarified in the course of the proof.
To this end, we introduce a suitable system of coordinates around the embedded, coorientable loop~$\gc$ in $M$, in which the conditions from \Cref{def: semi strong magnetic systems} can be expressed entirely locally; see \Cref{lemma: equivalent conditions for weak geodesic type}. 

From now on, fix a coorientable embedded loop \( \g \) in \( M \) of period \( T > 0 \). For brevity, we set \( \SS^1 := \mathbb{R} / T \mathbb{Z} \); we will not mention \( T \) explicitly again.
We consider the normal bundle~$\nu_\gamma$ of the embedded submanifold $\gamma(\SS^1) \subset M$. Recall from Subsection \ref{subsection: Main Results} that coorientability of the loop $\g$ just means that the pullback bundle $\g^*TM$ over $\SS^1$ is orientable.
The splitting of vector bundles
\[
\gamma^*TM = T\SS^1 \oplus \nu_\gamma 
\]
shows that orientability of $\g^*TM$ and $\nu_\g $ are equivalent. As every orientable vector bundle over the circle is trivial (see~\cite[Prop.~23.14]{Bott-Tu}), we conclude that $\nu_\g$ is a trivial bundle.\\
Therefore, we can choose a tubular neighborhood $U_\g$ of $\g$ in $M$ of the form
\begin{equation}\label{eq: tub neib U gamma}
    U_\g \cong \SS^1 \times \IR^{m-1}, \quad \text{where } m := \dim(M),
\end{equation}
in which the loop $\g$ is given in local coordinates by
\begin{equation}\label{eq: gamma in local coord in tub nbgh}
    \g(t) = (t, \mathbf{0}) = (t, 0, \ldots, 0) \in \SS^1 \times \IR^{m-1}, \qquad \text{for all } t \in \SS^1.
\end{equation}
The image of $\g$ corresponds to the zero section $\SS^1 \times \{\mathbf{0}\}$ of the tubular neighborhood. Moreover, in the local coordinates given by \eqref{eq: gamma in local coord in tub nbgh}, the derivative of $\g$ is constant and given by
\begin{equation}\label{eq: deriative of gamma in tub neighb}
    \dot \g(t) = e_1 = (1, \mathbf{0}) \in \IR^m, \qquad \forall \, t \in \SS^1 \,.
\end{equation}
Let $g$ be a Riemannian metric on the tubular neighborhood $U_\g$ of $\gc$, as given in~\eqref{eq: tub neib U gamma}. From now on, we identify $U_\g$ with $\SS^1 \times \IR^{m-1}$ without further comment. The metric $g$ can then be described by a smooth map
\[
G = (g_{ij}) \colon \SS^1 \times \IR^{m-1} \rightarrow \IR^{m \times m},
\]
taking values in the space of symmetric positive-definite matrices. To make this precise, let $\langle \cdot , \cdot \rangle$ denote the standard Euclidean inner product on $\IR^m$.\\ Then, for all $(t, \mathbf{x}) = (t, x_2, \ldots, x_m) \in \SS^1 \times \IR^{m-1}$ and all $v_1, v_2 \in \IR^m$, the metric $g$ satisfies
\begin{equation}\label{eq: repres of g in Ugc through G}
    g_{(t, \mathbf{x})}(v_1, v_2) = \langle G(t, \mathbf{x}) \cdot v_1, v_2 \rangle.
\end{equation}
Similarly, let $\alpha$ be a $1$-form on $\SS^1 \times \IR^{m-1}$. We denote by $V^\alpha$ the metric dual of $\alpha$ with respect to the Euclidean inner product. That is,
\begin{equation}\label{eq: expression alpha through V}
    \alpha_{(t, \mathbf{x})}(v) = \langle V^\alpha(t, \mathbf{x}), v \rangle \qquad \forall\, (t, \mathbf{x}) \in \SS^1 \times \IR^{m-1}, \; v \in \IR^m.
\end{equation}
With this notation in place, we can now state the following: 
\begin{lem}
\label{lemma: equivalent conditions for weak geodesic type}
In the notation introduced above, let $\gamma$ be a smooth coorientable embedded loop in $M$ with tubular neighborhood $U_\gamma$ as in~\eqref{eq: tub neib U gamma}, equipped with a Riemannian metric $g$ (represented as in~\eqref{eq: repres of g in Ugc through G} by $G$) and a $1$-form $\alpha$ (represented as in~\eqref{eq: expression alpha through V} by its Euclidean metric dual $V^\alpha$).\\
Then the following statements hold:
\begin{enumerate}
    \item\label{it: dual} The Euclidean metric dual of $\alpha$ along $\gc$ is $\dot\gamma$, that is,
    \[
    \alpha_{\gamma(t)}(\cdot) = g_{\gamma(t)}(\dot\gamma(t), \cdot) \qquad \forall\, t \in \SS^1 \, ,
    \]
    if and only if the first column of $G$ coincides with $V^\alpha$ along $\SS^1 \times \{\mathbf{0}\}$, that is
    \[
    V^\alpha(t, \mathbf{0}) = G(t, \mathbf{0}) \cdot e_1 \qquad \forall\, t \in \SS^1 \, ,
    \]
    where $e_1$ denotes the first standard unit vector.
    \item\label{it: direction} The velocity vector $\dot\gamma$ lies in the kernel of $\mathrm{d}\alpha$, that is,
    \[
    \dot\gamma(t) \in \ker \mathrm{d}\alpha_{\gamma(t)} \qquad \forall\, t \in \SS^1 \, ,
    \]
    if and only if $V^\a$ satisfies 
    \[
    \partial_1 V^\alpha_\ell = \partial_\ell V^\alpha_1 \qquad \text{along } \SS^1 \times \{\mathbf{0}\} \, , \quad \forall\, \ell = 1, \ldots, m \, .
    \]
    \item\label{it: condition on geodesic} The loop $\gamma$ is a geodesic of $(M, g)$ if and only if
    \begin{equation}
    \label{eq: geodesic equivalent condition}
        0 = 2\, \partial_1 g_{\ell 1} - \partial_\ell g_{11} \qquad \text{along } \SS^1 \times \{\mathbf{0}\} \, , \quad \forall\, \ell = 1, \ldots, m \, .
    \end{equation}
\end{enumerate}
\end{lem}

\begin{proof}
    For \ref{it: dual}, using~\eqref{eq: deriative of gamma in tub neighb} and~\eqref{eq: repres of g in Ugc through G}, we compute:
\[
g_{\gamma(t)}(\dot\gamma(t), v) = \langle G(\gamma(t)) \cdot \dot\gamma(t), v \rangle = \langle G(t, \mathbf{0}) \cdot e_1, v \rangle \qquad \forall\, t \in \SS^1 \, , \; v \in \IR^m
\]
Combining this with~\eqref{eq: expression alpha through V}, we see that
\[
\alpha_{\gamma(t)}(v) = g_{\gamma(t)}(\dot\gamma(t), v) \qquad \forall\, t \in \SS^1 \, , \; v \in \IR^m
\]
holds if and only if
\[
\langle V^\alpha(t, \mathbf{0}), v \rangle = \langle G(t, \mathbf{0}) \cdot e_1, v \rangle \qquad \forall\, t \in \SS^1 \, , \; v \in \IR^m \, .
\]
Since the Euclidean inner product $\langle \cdot, \cdot \rangle$ is nondegenerate, this equality holds for all \( v \in \IR^m \) if and only if
\[
V^\alpha(t, \mathbf{0}) = G(t, \mathbf{0}) \cdot e_1 \qquad \forall\, t \in \SS^1 \, .
\]
This completes the proof of \ref{it: dual}.
    
\noindent For \ref{it: direction},  by~\eqref{eq: expression alpha through V}, the exterior derivative of the $1$-form $\alpha$ is given by
\begin{align*}
    \mathrm{d} \alpha = \sum_{k < \ell} \left( \partial_k V^\alpha_\ell - \partial_\ell V^\alpha_k \right) \, \mathrm{d}x^k \wedge \mathrm{d}x^\ell.
\end{align*}
The claim then follows by inserting $\dot \g (t) = e_1 $ at $\g(t) = (t, \mathbf{0})$, see again \eqref{eq: gamma in local coord in tub nbgh} and \eqref{eq: deriative of gamma in tub neighb}.\\
    \noindent We now prove \ref{it: condition on geodesic}. The curve $\g = (\gamma_1, \ldots , \gamma_m)$  is a geodesic of $(M,g)$
    if and only if it is a solution of the geodesic equation
    \begin{equation}\label{eq: geodesic equation in local coordiantes}
        \ddot \g_k + \sum_{i,j} \Gamma^k_{ij}(\g) \, \dot \g_i \, \dot \g_j = 0 \qquad \forall \, k=1, \ldots , m \, .
    \end{equation}
    where $\Gamma_{ij}^k$ denote the Christoffel symbols of the Levi-Civita connection of $g$ in the standard basis of $\IR^m \,$.
    Using $\eqref{eq: deriative of gamma in tub neighb}$ we see $\ddot\gamma=0$ and $\dot \g_j = \delta_{1j}$. Hence \eqref{eq: geodesic equation in local coordiantes} reduces to
    \[
    \Gamma^k_{11}(\gamma(t)) = 0 \quad \forall \, t\in \SS^1 \, , \, \forall\, k = 1, \ldots, m\, ,
    \]
    and thus
    \begin{equation}\label{eq: gc_11 equal zero}
        \Gamma^k_{11} = 0 \qquad \text{ along } \gamma(\SS^1 ) = \SS^1 \times \{\mathbf{0}\} \qquad \forall \, k=1, \ldots , m \, .
    \end{equation}
    Using the following standard formula for the Christoffel symbols 
    \begin{align*}
        (\Gamma^k_{11})_{k=1}^m = \frac{1}{2} \, G^{-1} \cdot \left( \partial_1 \, g_{\ell 1} + \partial_1 \, g_{\ell 1} - \partial_\ell \, g_{11}  \right)_{\ell=1}^m \, ,
    \end{align*}
    (where the vectors on the left- and right-hand side are column vectors respectively), we conclude that ~\eqref{eq: gc_11 equal zero} is equivalent to
    \begin{align*}
        0 = \partial_1 \, g_{\ell 1} + \partial_1 \, g_{\ell 1} - \partial_\ell \, g_{11} \qquad \text{ along } \SS^1 \times \{\mathbf{0}\} \qquad \forall \, \ell=1, \ldots , m \, .
    \end{align*}
    This completes the proof of \emph{\ref{it: condition on geodesic}}.
\end{proof} 
With this preparation---which, as we just have seen, allows us to reformulate all relevant conditions in convenient local coordinates---we now turn to the proof of \Cref{prop: existence of semi strong pair}.
\begin{proof}[Proof of \Cref{prop: existence of semi strong pair}]
As already mentioned, the conditions in \Cref{def: semi strong magnetic systems} are local in nature. Therefore, we first construct the $1$-form~$\alpha$ in~\ref{it: 1 prop existnece of semi strong pair} and the Riemannian metric~$g$ in~\ref{it: 2 prop existnece of semi strong pair} locally. We then explain, at the end of the first construction, how this local model can be extended to a global one. The second case is omitted, as the extension follows line by line from the first.\\
We begin with \ref{it: 1 prop existnece of semi strong pair}. Fix coordinates $\SS^1 \times \IR^{m-1}$ as in~\eqref{eq: tub neib U gamma} and denote by $G=(g_{ij})$ the matrix representation of the metric $g$ in these coordinates, see \eqref{eq: repres of g in Ugc through G}.
Let us define a vector field 
\[
V^\a = (V_1^\a , \ldots , V_m^\a) : \SS^1 \times \IR^{m-1} \rightarrow \IR^m \, .
\]
We let $V^\a_1$ at $(t, \mathbf{x}) =  (t, x_2, \ldots , x_m )\in \SS^1 \times \RR^{m-1}$ be given by
\begin{equation}\label{eq: first column of V_1}
    V_1^\a(t, \mathbf{x}) := g_{11}(t, \mathbf{0}) + \sum_{\ell=2}^m \,\partial_1\,  g_{\ell 1} (t, \mathbf{0}) \, x_\ell \, , 
\end{equation}
and for all other coordinate entries $\ell=2, \ldots , m$ we set
\begin{equation}\label{eq: defi V_l for l geq 2}
    V_\ell^\a (t, \mathbf{x}) := g_{\ell 1}(t, \mathbf{0}) \qquad \forall \, t\in \SS^1 \, , \, \mathbf{x} \in \RR^{m-1}\, .
\end{equation}
By its definition, along $\SS^1 \times \{\mathbf{0}\}$ the vector field $V^\a$ is the first column of $G$, that is 
\begin{equation}\label{eq: V alpha is first column of G}
    V^\alpha(t, \mathbf{0}) = G(t, \mathbf{0}) \cdot e_1 \qquad \forall\, t \in \SS^1 \, .
\end{equation}
Furthermore it follows again directly from ~\eqref{eq: first column of V_1} and \eqref{eq: defi V_l for l geq 2} that along $\SS^1 \times \{\mathbf{0}\}$ it holds that
\begin{equation}\label{eq: commuting partial derivaties for V a}
    \partial_1 \, V^\a_\ell = \partial_\ell \, V_1^\a \qquad \forall \ell = 1, \ldots , m \, .
\end{equation}
The $1$-form~$\alpha$ on $\SS^1 \times \RR^{m-1}$ that corresponds to the vector field $V^\a$ via~\eqref{eq: expression alpha through V} now has the desired properties, as can be seen from parts~(\ref{it: dual}) and (\ref{it: direction}) of Lemma \ref{lemma: equivalent conditions for weak geodesic type} and \eqref{eq: V alpha is first column of G} and \eqref{eq: commuting partial derivaties for V a}.
This completes the local construction required for~\ref{it: 1 prop existnece of semi strong pair}.

The global extension proceeds as follows: Fix a small compact neighborhood \( K\) of the loop~$\gamma$ contained in its tubular neighborhood $\SS^1\times \RR^{m-1}$, and choose a smooth bump function \( \rho : M \to [0,1] \) with $\rho=1$ on $K$ and support contained in $\SS^1 \times \IR^{m-1} \,$. Next, choose a smaller compact neighborhood \( K_1 \subseteq \mathrm{Int}(K) \) of $\g$ and a $1$-form \( \tilde\alpha \) on \( M \setminus K_1 \). Then the interpolated $1$-form
\[
\bar\alpha := (1 - \rho)\, \tilde\alpha + \rho\, \alpha
\]
defines a smooth global $1$-form on \( M \) which coincides with \( \alpha \) on \( K \), and thus retains the desired local properties. Which finishes the construction of \ref{it: 1 prop existnece of semi strong pair}. \\

\noindent
We now turn to the proof of part~\ref{it: 2 prop existnece of semi strong pair}. While the argument follows the same general strategy as in part~\ref{it: 1 prop existnece of semi strong pair}, it requires some additional technical considerations. First observe that without loss of generality we may assume that
    \begin{equation}\label{eq: prequisites it 2 on alpha}
         \dot \g(t) \in \ker \d \a_{\g(t) } \qquad\text{and} \qquad \a_{\g(t)}(\dot \g(t)) = 1 \qquad \forall \, t\in \SS^1 \,.
    \end{equation}
Indeed, if more generally we only have $\dot \g \in \ker \d \a_\g$ and that $\a_\g(\dot \g)$ is constant with value $r > 0 \,$, then the reparametrized curve $\g(\frac{\cdot}{r})$ satisfies \eqref{eq: prequisites it 2 on alpha} and by assumption therefore admits an infinite dimensional space of Riemannian metrics $g$ so that $\g(\frac{\cdot}{r})$ is a magnetic geodesic of semi-strong geodesic type of $(M,g, \d\a) \,$. Then, for each such $g$ the original loop $\g$ is a magnetic geodesic of semi-strong geodesic type of $(M , \, r^{-1} \, g \, , \d \a) \,$.\\
In the tubular neighborhood $\SS^1\times \RR^{m-1}$ around $\gc $ we define the Euclidean metric dual $V^\a$ of $\a$ as in~\eqref{eq: expression alpha through V}. Using \eqref{eq: deriative of gamma in tub neighb} and \eqref{eq: prequisites it 2 on alpha}, we can conclude
\begin{equation}\label{eq: V^alpha_1 (t) = 1}
    V^\a_1(t, \mathbf{0}) = \langle  V^\a (t, \mathbf{ 0}) \, , \, e_1 \rangle =  \a_{\g(t)}(\dot \g(t)) = 1  \qquad \forall \, t \in \SS^1 \,\, . 
\end{equation}
 Next, we introduce a loop of symmetric positive-definite matrices associated to the vector field \( V^\alpha \), defined by
\begin{equation}\label{eq: def of G(t) inside proof}
    \widetilde G (t) := B(t)^{\mathrm{T}} B(t) \qquad \forall\, t\in \SS^1 \, ,
\end{equation}
where \( B(t) \) is a loop of invertible \( m \times m \) matrices given by
\[
B(t) := \begin{pmatrix}
1 & V_2^\a(t, \mathbf{0}) & V_3^\a(t, \mathbf{0}) & \cdots & V_m^\a(t, \mathbf{0}) \\
0 & 1 & 0 & \cdots & 0 \\
0 & 0 & 1 & \cdots & 0 \\
\vdots & & & \ddots & \vdots \\
0 & 0 & 0 & \cdots & 1
\end{pmatrix}.
\] 
Using \eqref{eq: V^alpha_1 (t) = 1} and the definition of $ \widetilde G$ in \eqref{eq: def of G(t) inside proof}, observe that the first column of $\widetilde G(t)$ is given by $V^\a(t, \mathbf{0}) $. Moreover $\widetilde G(t)$ is positive definite as $B(t)$ is invertible.\\
Next, we construct a smooth extension \( G \) of \(  \widetilde{G} \) to a small neighborhood of \( \SS^1 \times \{ \mathbf{0} \} \) within the tubular neighborhood \( \SS^1 \times \RR^{m-1} \) of \( \gamma \), such that this extension satisfies the conditions stated in \Cref{lemma: equivalent conditions for weak geodesic type}.
To this end we define the component functions $g_{ij}$ of $ G$ by
\begin{align}\label{eq: definition tilde G in proof}
\begin{cases}
     g_{11}(t, \mathbf{x}) := \widetilde{g}_{11}(t) + 2\, \sum_{\ell=2}^m (\frac{\d }{\d t}  \widetilde{g}_{\ell 1} (t) ) \, x_\ell   & \\
    g_{ij}(t, \mathbf{x}) := \widetilde{g}_{ij}(t) &  \text{ for } (i,j) \not= (1,1) \,\, ,
\end{cases}
\end{align}
for  each $(t, \mathbf{x}) = (t, x_2, \ldots , x_m) \in \SS^1 \times \IR^{m-1}$. Since $ \widetilde{G}(t)$ is symmetric, also $G(t, \mathbf{x})$ is. Additionally, because $\widetilde{G}$ takes values in the space of positive definite matrices—and since positive definiteness is an open condition—there exists an open neighborhood \begin{equation}\label{eq: defi U-1}
    U_1 \subseteq \SS^1 \times \IR^{m-1}
\end{equation} 
of $\SS^1 \times \{\mathbf{0}\}$ on which the function $G$ also takes values in the positive definite matrices.
Furthermore, it follows directly from the definition of \( G \) in~\eqref{eq: definition tilde G in proof} that for all \( \ell = 2, \dots, m \), we have
\begin{equation*}\label{eq: tilde G sat cond }
    \partial_\ell \, g_{11}(t, \mathbf{0}) = 2 \, \frac{\mathrm{d} \widetilde{g}_{\ell 1}}{\mathrm{d}t}(t) = 2 \, \partial_1 \, g_{\ell 1}(t, \mathbf{0}) \qquad \forall\, t \in \SS^1 \, ,
\end{equation*}
For $\ell=1$ we use the definition in \eqref{eq: definition tilde G in proof}, the fact that the first column of $ \widetilde{G}(t)$ is given by $V^\a(t, \mathbf{0})$ for all $t$ and ~\eqref{eq: V^alpha_1 (t) = 1} to compute
\[
\partial_1 \, g_{11}(t, \mathbf{0}) = \frac{\mathrm{d} \widetilde{g}_{11}}{\mathrm{d}t}(t) = \partial_1 \, V^\alpha_1(t, \mathbf{0}) = 0 \, .
\]
Combining our considerations for $\ell=2, \ldots , m$ and $\ell=1$ we have thus shown
\begin{equation}\label{eq: tilde G sat lemma condi 3 real eq}
    0 = 2 \, \partial_1 \, g_{\ell 1}(t, \mathbf{0}) - \partial_\ell \, g_{11}(t, \mathbf{0}) \qquad \forall\, t \in \SS^1.
\end{equation}
Now we define a Riemannian metric $g$ on $U_1$ via \eqref{eq: repres of g in Ugc through G}, that is
\begin{align*}
g_{(t, \mathbf{x})} (v_1, v_2) := \langle G(t, \mathbf{x}) \cdot v_1 \, , \, v_2 \rangle \qquad \forall \, (t, \mathbf{x})  \in U_1 \subseteq   \SS^1\times \RR^{m-1} \text{ and } v_1, v_2 \in \IR^m \,\, .
\end{align*}
Using that the first column of $ G(t, \mathbf{0})= \widetilde{G}(t)$ is $V^\a(t, \mathbf{0}) $ and \eqref{eq: tilde G sat lemma condi 3 real eq}, parts \ref{it: dual} and \ref{it: condition on geodesic} in Lemma \Cref{lemma: equivalent conditions for weak geodesic type} now imply that $\g$ is a geodesic of $(U_1, g)$ and also that
\[
\a_{\gc(t)} (v) = g_{\gc(t)}(\dot \g(t) , \, v \,) \qquad \forall \, t\in \SS^1 \,, \,  v\in T_{\gc(t)}M\, .
\]
Hence, \( \gamma \) is a magnetic geodesic of semi-strong geodesic type on \( (U_1, g, \mathrm{d}\alpha) \), where \( U_1 \) is the neighborhood of $\g$ given in~\eqref{eq: defi U-1}.
This finishes the local construction.

Finally, by an argument similar to that used at the end of part~\ref{it: 1 prop existnece of semi strong pair}, the metric \( g \) can be extended smoothly to all of \( M \), completing the proof.
\end{proof}
\subsection{Rescaling of exact magnetic systems of semi-strong geodesic type }\label{ss: rescaling} In order to prove \Cref{prop: rescaling of magnetic systems of semi strong geodesic type}, we show one technical key lemma, and that magnetic geodesics of semi-strong geodesic type are preserved under a certain class of localized rescalings of the metric and magnetic field.
\begin{lem}[Key Lemma]
\label{lemma: derivative of |alpha|^2_g}
Suppose $\g$ is a magnetic geodesic of semi-strong geodesic type of $(M,g , \d \a) \,$. Consider the function 
\[\varrho := |\a|_g^2 : M \rightarrow [0,\infty) \, .\] 
Then the following holds: \begin{enumerate}
    \item \label{it: 1 key lemma}The function $\varrho$ is constant with positive value along $\gc$.
    \item \label{it: 2 key lemma} The differential $\d\varrho$ vanishes along $\gc$ that is 
    \[
    \d \varrho_{\g(t)}(v)=0 \qquad \forall \, t\in \SS^1 \,, \,  v\in T_{\gc(t)}M\, .
    \]
\end{enumerate}
\end{lem}
\begin{proof}
(\ref{it: 1 key lemma}) follows immediately from~\eqref{eq: energy conservation along gc}.\\
For (\ref{it: 2 key lemma}), we compute the differential $\mathrm{d}\varrho$ along $\gc$. For this, recall the following notation: $U_\gamma \cong \mathbb{S}^1 \times \mathbb{R}^{m-1}$ denotes the tubular neighborhood of $\gc$ as in~\eqref{eq: tub neib U gamma}; $G = (g_{ij})_{ij}$ is the matrix representation of the metric $g$ in $U_\gc$, as in~\eqref{eq: repres of g in Ugc through G}; and $V^\alpha$ is the Euclidean metric dual of $\alpha$ in $U_\gc$, as given in~\eqref{eq: expression alpha through V}. 

By using~\eqref{eq: repres of g in Ugc through G} and~\eqref{eq: expression alpha through V}, we see that the vector field $G^{-1} \cdot V^\alpha$ is the $g$-dual of $\alpha$, that is,
\begin{align*}
    \alpha_{(t,\mathbf{x})}(v) = g_{(t,\mathbf{x})} \big( G^{-1}(t,\mathbf{x}) \cdot V^\alpha(t,\mathbf{x}), \, v \big)  
    \qquad \forall (t,\mathbf{x}) \in \mathbb{S}^1 \times \mathbb{R}^{m-1}, \; v \in \mathbb{R}^m.
\end{align*}
From this and the definition of $\varrho$, we conclude that $\varrho$ has the following expression in the tubular neighborhood $\SS^1\times\RR^{m-1}$:
\begin{equation}
\label{eq: u equivalent def}
    \varrho(t,\mathbf{x}) = \left\langle V^\alpha(t,\mathbf{x}), \, G^{-1}(t,\mathbf{x}) \cdot V^\alpha(t,\mathbf{x}) \right\rangle 
    \quad \forall (t,\mathbf{x}) \in \mathbb{S}^1 \times \mathbb{R}^{m-1}.
\end{equation}
Thus, to compute the differential of $\varrho$, we need to differentiate the right-hand side of~\eqref{eq: u equivalent def}. In particular, we must compute the differential of the map $(t,\mathbf{x}) \mapsto G^{-1}(t,\mathbf{x})$. This is done by differentiating the identity $G \cdot G^{-1} = \mathbbm{1}$ using the Leibniz rule and then rearranging terms, which yields
\begin{equation}
    \label{eq: derivative G^(-1)}
    \mathrm{d}(G^{-1}) = - G^{-1} \cdot (\mathrm{d}G) \cdot G^{-1} \,.
\end{equation}
Differentiating the right-hand side of~\eqref{eq: u equivalent def} using the Leibniz rule for bilinear pairings, and applying~\eqref{eq: derivative G^(-1)} as well as the symmetry of $G^{-1}$, we obtain for all $(t,\mathbf{x})\in\SS^1\times \RR^{m-1}$ and all $v\in \RR^m$
\begin{equation}\label{eq: differential varrho in key lemma}
    \begin{split}
\mathrm{d}\varrho_{(t,\mathbf{x})} [v]
&= 2 \left\langle (\mathrm{d}V^\alpha)(t,\mathbf{x})[v] \, , \, G^{-1}(t,\mathbf{x}) \cdot V^\alpha(t,\mathbf{x}) \right\rangle \\
&- \left\langle G^{-1}(t,\mathbf{x}) \cdot V^\alpha(t,\mathbf{x}) \, , \, (\mathrm{d}G)(t,\mathbf{x})[v] \cdot G^{-1}(t,\mathbf{x}) \cdot V^\alpha(t,\mathbf{x}) \right\rangle
\end{split}
\end{equation}
Since $\gamma$ is of semi-strong geodesic type for $(M, g, \mathrm{d}\alpha)$, by \Cref{lemma: equivalent conditions for weak geodesic type}~(\ref{it: dual}), we have
\begin{equation}
    G(t, \mathbf{0})^{-1} \cdot V^\alpha(t, \mathbf{0}) = e_1 \quad \forall \, t \in \SS^1 \,.
\end{equation}
From this, and using~\eqref{eq: differential varrho in key lemma}, we conclude that the differential of $\varrho$ along $\gamma$ is given by
\begin{equation}\label{eq: differential varrho along gc e_1}
    \mathrm{d}\varrho_{\gamma(t)}[v] = 2 \, \left\langle (\mathrm{d}V^\alpha)_{\gamma(t)}[v], \, e_1 \right\rangle 
    - \left\langle e_1, \, (\mathrm{d}G)_{\gamma(t)}[v] \cdot e_1 \right\rangle 
    = 2 \, (\mathrm{d}V^\alpha_1)_{\gamma(t)}[v] - (\mathrm{d}g_{11})_{\gamma(t)}[v]
\end{equation}
for all \( t \in \SS^1 \) and \( v \in \RR^m  \).\\
Therefore, the differential of $\varrho$ vanishes along $\gamma$ if and only if the right-hand side of~\eqref{eq: differential varrho along gc e_1} vanishes, that is
\begin{equation}\label{eq: du=0 equivalent condition}
    (\mathrm{d}g_{11})_{\gamma(t)}[v] = 2 \, (\mathrm{d}V^\alpha_1)_{\gamma(t)}[v] 
    \quad \forall\, t \in \SS^1 , \quad \forall v \in \RR^m \,.
\end{equation}
This identity indeed holds. Since $\gamma$ is a magnetic geodesic of semi-strong geodesic type for $(M, g, \mathrm{d}\alpha)$, we can successively apply parts~(\ref{it: condition on geodesic}), (\ref{it: dual}), and (\ref{it: direction}) of \Cref{lemma: equivalent conditions for weak geodesic type} to deduce the following chain of equalities:
\begin{equation*}
    \partial_\ell \, g_{11}(t, \mathbf{0}) 
    \overset{\text{(\ref{it: condition on geodesic})}}{=} 2 \, \partial_1 \, g_{\ell 1}(t, \mathbf{0}) 
    \overset{\text{(\ref{it: dual})}}{=} 2 \, \partial_1 \, V^\alpha_\ell (t, \mathbf{0}) 
    \overset{\text{(\ref{it: direction})}}{=} 2 \, \partial_\ell \, V^\alpha_1(t, \mathbf{0}) 
    \quad \forall\, t \in \SS^1 \, , \quad \forall\, \ell = 1, \ldots, m\,.
\end{equation*}
This proves~\eqref{eq: du=0 equivalent condition} and completes the proof of the key lemma.
\end{proof}
Next, we show that exact magnetic systems of semi-strong geodesic type are preserved under a class of localized rescalings. These rescalings are applied to the metric and the magnetic field, and are required to exhibit the same local behavior along $\gamma$ as the function considered in \Cref{lemma: derivative of |alpha|^2_g}. 

More precisely, the function used for rescaling must satisfy conditions~(\ref{it: 1 key lemma}) and~(\ref{it: 2 key lemma}) stated therein. Once this is ensured, the localized rescaling preserves the semi-strong geodesic type of the magnetic system.
\begin{lem}
\label{lemma: rescale adapted tuple}
Let $\gamma$ be a magnetic geodesic of semi-strong geodesic type for the system $(M, g, \mathrm{d}\alpha)$. Let
\[
\varrho : M \rightarrow (0,\infty)
\]
be a smooth positive function with
\[
\varrho(\gamma(t)) = 1 \quad \text{and} \quad \mathrm{d}\varrho_{\gamma(t)}[v] = 0 
\quad \forall \, t \in \SS^1 \, , \quad \forall \, v \in T_{\gamma(t)}M \, .
\]
Then $\gamma$ is a magnetic geodesic of semi-strong geodesic type for both systems $
(M, \, \varrho g, \, \mathrm{d}\alpha) \quad$ and $\quad (M, \, g, \, \mathrm{d}(\varrho\alpha))$.
In particular, both magnetic systems \( (M, \varrho\,  g, \mathrm{d}\alpha) \) and \( (M, g, \mathrm{d}(\varrho \a))) \)  are of semi-strong geodesic type.
\end{lem}
\begin{proof}
    We begin by noting that the condition 
    \begin{equation}\label{eq: varrho constantly one along gc}
    \varrho(\gc(t))=1 \quad \forall \, t\in \SS^1
    \end{equation}
 ensures that for the rescaling $\varrho\cdot g$ of the metric \( g \) the one-form $\alpha$ along $\gc$  is the  $\varrho\cdot g$ metric dual of  the velocity $\dgc$. The same conclusion holds also for the rescaling $\varrho\cdot \a$ of the one-form \( \a \),  that is the one-form $\varrho \cdot \alpha$ along $\gc$ is the  $ g$-metric dual of  the velocity $\dgc$. \\
In view of the definition of semi-strong geodesic type, it now suffices to show that $\g$ is also a magnetic geodesic of geodesic type of both $(M , \, \varrho \, g \, , \d \a)$ and $(M , g , \d (\varrho \, \a)) \,$. 

We begin with the first of the two. Since $\gc$ is a magnetic geodesic of semi-strong geodesic type of $(M, g)$, it is also of geodesic type, and thus by \Cref{Def: magnetic systems of n-geodesic type} it holds that
\begin{equation}\label{eq: proof rescaling dgc in kern of d alphja}
    \dgc(t) \in \ker \d\a_{\gc(t)} \quad \forall \, t \in \SS^1 \, .
\end{equation} 
So it remains to prove that $\gc$ is a geodesic of $(M,\varrho g)$. 

For this, we choose the tubular neighborhood $U_\g$ as in \eqref{eq: tub neib U gamma} and let $G=(g_{ij})_{ij}$ denote the matrix representing $g$ in these coordinates, see \eqref{eq: repres of g in Ugc through G}. At this point, we recall that $\gc$ is a magnetic geodesic of semi-strong geodesic type of $(M,g)$ and thus, by \Cref{def: semi strong magnetic systems}, a geodesic of $(M,g)$. 

By \ref{it: condition on geodesic} in Lemma \ref{lemma: equivalent conditions for weak geodesic type}, this implies that in the previously chosen tubular neighborhood it holds that 
\begin{equation}\label{eq: proof of lemma local coord for géodesics}
    2 \, \partial_1 \, g_{\ell 1}(t, \mathbf{0}) - \partial_\ell \, g_{11}(t, \mathbf{0}) = 0  
    \qquad \forall \, t \in \SS^1 \, .
\end{equation}
From this, and using the assumption that $\varrho$ has vanishing differential along $\g$, we derive
\begin{equation}\label{eq: geodesic equation of g rho}
    2 \, \partial_1 (\varrho \, g_{\ell 1})(t, \mathbf{0}) - \partial_\ell (\varrho \, g_{11})(t, \mathbf{0}) 
    = 2 \, \partial_1 \, g_{\ell 1}(t, \mathbf{0}) - \partial_\ell \, g_{11}(t, \mathbf{0}) = 0  
    \qquad \forall \, t \in \SS^1 
\end{equation}
As $\varrho \, g$ is represented by the matrix $(\varrho \, g_{ij})_{ij}$, from \eqref{eq: geodesic equation of g rho} by using again part \ref{it: condition on geodesic} in \Cref{lemma: equivalent conditions for weak geodesic type} we can conclude that $\g$ is a geodesic of $(M, \varrho \, g) \,$.

It remains to show that $\gc$ is a magnetic geodesic of semi-strong geodesic type of $(M, g, \varrho \cdot \d\a)$. By the discussion above, it suffices to prove that $\dgc$ lies in the kernel of $\d(\varrho\, \a)$ along $\gc$. \\
By computing the differential of $\varrho\, \a$, and using that the differential of $\varrho$ vanishes along $\g$ together with \eqref{eq: varrho constantly one along gc}, we derive the following chain of equalities:
\[
\d (\varrho \a)_{\gc(t)} = \d \varrho_{\gc(t)} \wedge \a_{\gc(t)} + \varrho(\gc(t)) \, \d \a_{\gc(t)} = \d \a_{\gc(t)} \quad \forall \, t \in \SS^1 \, ,
\]
from which we conclude by using~\eqref{eq: proof rescaling dgc in kern of d alphja} that 
\[
\dot \g(t) \in \ker \d (\varrho \a)_{\g(t)} \quad \forall \, t \in \SS^1.
\]
This finishes the proof.
\end{proof}

We close this subsection by proving \Cref{prop: rescaling of magnetic systems of semi strong geodesic type}.

\begin{proof}[Proof of \Cref{prop: rescaling of magnetic systems of semi strong geodesic type}]
Let \( \gamma \) be a magnetic geodesic of semi-strong geodesic type of \( (M, g, \mathrm{d}\alpha) \), where we assume without loss of generality that \( \gamma \) is parametrized with unit speed. This implies, via~\eqref{eq: energy conservation along gc}, that
\begin{equation}
\label{eq: wlog unit speed}
    |\a_\gc|_g^2=| \dot{\gamma} |_g^2 = \alpha_\gamma(\dot{\gamma}) = 1 \, .
\end{equation}

Before beginning the proof, we briefly explain why the general case reduces to this one.

\medskip

In part~(\ref{it 1: rescaling}), we assume there exists a Riemannian metric \( \tilde{g} = \varrho_1 \, g \) and a $1$-form \( \alpha \), such that \( \gamma \) is a unit-speed magnetic geodesic of semi-strong geodesic type of \( (M, \tilde{g}, \mathrm{d}\alpha) \), with \( |\alpha|_{\tilde{g}} \) attaining its maximum along \( \gamma \). Then, for each \( r > 0 \), the constant-speed reparametrization \( \gamma_r \), defined as in~\eqref{eq: definition gamma_r}, is a magnetic geodesic of semi-strong geodesic type of \( (M, \tilde{g}, \mathrm{d}(r \alpha)) \), with \( |r \alpha|_{\tilde{g}} \) being maximal along \( \gamma_r \).\\
Considering instead the reparametrization \( \gamma_{r^2} \) and the rescaling $r^{-2} \, g$ of the metric, the argument in part~(\ref{it 2: rescaling}) follows analogously to that of~(\ref{it 1: rescaling}).

\medskip
From now on, let \( \gamma \) be a unit-speed magnetic geodesic of semi-strong geodesic type of \( (M, g , \mathrm{d}\alpha) \). 
Consider the function
\begin{equation}\label{eq: def rho in proof of prop 4.3}
    \varrho := |\alpha|_g^2 : M \rightarrow [0,\infty) \, .
\end{equation}
By~\eqref{eq: wlog unit speed}, it holds that
\begin{equation}\label{eq: proof rho const gc one proff prop 43}
    \varrho(\gamma(t)) = 1 \qquad \forall \, t \in \SS^1 \, .
\end{equation}
Let us fix an open neighborhood \( \Omega \subset M \) of \( \gamma \) on which \( \varrho \) is positive. By \Cref{lemma: derivative of |alpha|^2_g}, the function \( \varrho \) (and hence also \( (\sqrt{\varrho})^{-1} \)) has vanishing differential along \( \gamma \).

This, together with~\eqref{eq: def rho in proof of prop 4.3}, allows us to apply \Cref{lemma: rescale adapted tuple} and conclude that \( \gamma \) is a magnetic geodesic of semi-strong geodesic type both of \( (\Omega , \, \varrho \, g , \mathrm{d}\alpha) \) and of \( (\Omega, g , \mathrm{d}((\sqrt{\varrho})^{-1} \alpha)) \).

Now fix a compact neighborhood $K \subseteq \Omega$ of $\g \,$. Since $\varrho$ is stricly positive on $\Omega \,$, we can choose a smooth strictly positive function $\widetilde\varrho : M \rightarrow (0,  \infty )$ so that
\begin{align}\label{eq: defi tilde rho}
        \begin{cases}
            \widetilde\varrho = \varrho   & \text{ on } K \\
            \widetilde\varrho \geq \varrho & \text{ on } M\setminus K \,\, .
        \end{cases}
\end{align}

Part (\ref{it 1: rescaling}): Consider the function \( \varrho_1 := \widetilde{\varrho} \) and the Riemannian metric \( \widetilde{g} := \varrho_1 \, g \). Since, by~\eqref{eq: defi tilde rho}, the metrics \( \widetilde{g} \) and \( \varrho \, g \) coincide on the neighborhood \( K \subseteq \Omega \), and since \( \gc \) is a magnetic geodesic of semi-strong geodesic type of \( (\Omega, \varrho\, g, \mathrm{d}\alpha) \), it follows that \( \gamma \) is also a magnetic geodesic of semi-strong geodesic type of \( (M, \widetilde{g}, \mathrm{d}\alpha) \).

Next, using the definition of \( \varrho \) in~\eqref{eq: def rho in proof of prop 4.3} and the definition of \( \widetilde{\varrho} = \varrho_1 \) in~\eqref{eq: defi tilde rho}, a short computation confirms the following relation between the \( \widetilde{g} \)-norm of \( \alpha \) and the \( g \)-norm of \( \alpha \):
\[
|\alpha|_{\widetilde{g}}^2 = |\alpha|_{(\widetilde{\varrho} g)}^2 = \frac{1}{\widetilde{\varrho}} \, |\alpha|_g^2 = \frac{\varrho}{\widetilde{\varrho}} \leq 1 \,.
\]
    From which we can conclude using \eqref{eq: def rho in proof of prop 4.3},\eqref{eq: proof rho const gc one proff prop 43} and \eqref{eq: defi tilde rho} that $\vert\a\vert_{\widetilde g}$ is maximal along $\g \,$.

Part (\ref{it 2: rescaling}): We begin with defining 
\begin{equation}\label{eq: def rho2 and tilde alpha}
    \varrho_2 := \frac{1}{\sqrt{\widetilde \varrho}}\quad\text{and}\quad \widetilde{\a} := \varrho_2 \, \a = \frac{1}{\sqrt{\widetilde{\varrho}}} \, \a \,. 
\end{equation}
By using \eqref{eq: defi tilde rho} and \eqref{eq: def rho2 and tilde alpha} we see that the one-forms $\widetilde{\a}$ and $\frac{1}{\sqrt{\varrho}} \, \a$ coincide on the neighborhood $K\subseteq \Omega$ of $\gc$. So it follows from the fact the $\gc$ is a magnetic geodesic of semi-strong geodesic type of $(\Omega,g,\d(\sqrt{\varrho}^{-1}\a))$  that $\g$ is also a magnetic geodesic of semi-strong geodesic type of $(M, g , \d\widetilde{\a}) \,$. By using \eqref{eq: defi tilde rho} and \eqref{eq: def rho2 and tilde alpha} we get the following chains of inequalities
\[
    |\widetilde{\a}|^2_g = \frac{1}{\widetilde{\varrho}} \, |\a|_g^2 = \frac{\varrho}{\widetilde{\varrho}} \leq 1 \, .
\]
   From which we can conclude by using~\eqref{eq: def rho in proof of prop 4.3},\eqref{eq: proof rho const gc one proff prop 43} and \eqref{eq: defi tilde rho} that $\vert\tilde\a\vert_{ g}$ is maximal along the magnetic geodesic $\g \,$.  Which finishes the proof. 
\end{proof}
\subsection{Proof of \Cref{Cor: multiplicity of magnetic geodesics of strong geodesic type}}
\label{ss: proof of Cor multiplicity}
We begin by proving that each of the coorientable, embedded, null-homologous loops \( \gamma_1, \dots, \gamma_n \) is a geodesic of \( (M, g) \) for some Riemannian metric \( g \). For simplicity of notation, we outline the construction for a single loop, say \( \gamma := \gamma_1 \). \\
Fix a tubular neighborhood \( U_\gamma \cong \mathbb{S}^1 \times \mathbb{R}^{m-1} \) of \( \gamma \), as in~\eqref{eq: tub neib U gamma}, and let \( K \subseteq U_\gamma \) be a compact neighborhood of \( \gamma \). Define \( g \) to be the Euclidean metric on \( K \) (in the chosen tubular coordinates), and extend it smoothly to all of \( M \). Then, by construction, \( \gamma \) satisfies~\eqref{eq: geodesic equivalent condition}, and hence, by~\ref{it: condition on geodesic} in \Cref{lemma: equivalent conditions for weak geodesic type}, the loop \( \gamma \) is a geodesic of \( (M, g) \). It is evident that this procedure can be carried out simultaneously for each of the coorientable, embedded, null-homologous loops \( \gamma_1, \dots, \gamma_n \).

Next, following the lines of \ref{it 2: rescaling} in \Cref{prop: rescaling of magnetic systems of semi strong geodesic type}, and performing the construction in~\eqref{eq: defi tilde rho} simultaneously around each \( \gamma_1, \dots, \gamma_n \), we can construct an infinite-dimensional space of \( 1 \)-forms \( \alpha \) on \( M \) such that the \( g \)-norm \( |\alpha|_g \) is maximal along each of the loops \( \gamma_1, \dots, \gamma_n \). By the remainder of the proof of \Cref{prop: rescaling of magnetic systems of semi strong geodesic type}, it follows that each of the loops \( \gamma_1, \dots, \gamma_n \) is a magnetic geodesic of semi-strong geodesic type of \( (M, g, \mathrm{d}\alpha) \), with \( |\alpha|_g \) maximal along each loop.

Since each \( \gamma_j \) for $j=1,\dots, n$ is null-homologous, we can apply \Cref{Cor: from semi strong to strong} to conclude that each of the loops \( \gamma_1, \dots, \gamma_n \) is in fact a magnetic geodesic of strong geodesic type of \( (M, g, \mathrm{d}\alpha) \). This completes the proof.

\appendix
\section{A criterion for contact type.}\label{Appendix A}
The following discussion is well known to experts, but is included here for the sake of completeness. 
In \cite[Prop. 2.4]{ConMacPat2004}, a criterion is given for an energy surface to be of contact type, formulated in terms of null-homologous invariant probability measures of the magnetic geodesic flow. A version of this result was previously proved by D. McDuff in \cite{McDuff1987}, based on Sullivan’s theory of structural cycles \cite{Sullivan1976}.

The criterion in \cite[Prop. 2.4]{ConMacPat2004} can be restated for energy surfaces of the magnetic geodesic flow as follows. 

\begin{prop}[\cite{ConMacPat2004}]\label{prop: for contact type}
Let $Z_E$ denote the Hamiltonian vector field generating the flow $\varPhi^t_{g,\mathrm{d}\alpha}$.
For a level of the energy $\kappa \in (0,\infty)$ of the magnetic geodesic flow $\varPhi^t_{g,\mathrm{d}\alpha}$ the following are equivalent:
\begin{enumerate}
    \item \label{it: prop invariant porb measure} For every null-homologous, $\varPhi^t_{g,\mathrm{d}\alpha}$-invariant probability measure $\mu$ on the energy surface $\Sigma_{\k}$ it holds that
    \[
    \int_{\Sigma_{\k}} \lambda(Z_E) \, \mathrm{d}\mu \neq 0.
    \]
    \item The energy surface $\Sigma_\kappa$ is of contact type.
\end{enumerate}
\end{prop}
\begin{rem}\label{r: defi nullhomologous prob measure}
Recall that a $\varPhi^t_{g,\mathrm{d}\alpha}$-invariant probability measure $\mu$ on $\Sigma_{\kappa}$ is said to be \emph{null-homologous} if its associated homology class $[\mu] \in H_1(\Sigma_{\kappa}, \mathbb{R})$ is zero. This class is defined via duality with cohomology by
\[
\langle [\mu], [\Theta] \rangle := \int_{\Sigma_\kappa} \Theta(Z_E) \, \mathrm{d}\mu \qquad \forall [\Theta] \in H^1(\Sigma_{\kappa}, \mathbb{R}) .
\]
\end{rem}
We will show that condition~(\ref{it: prop invariant porb measure}) in \Cref{prop: for contact type} fails if there exists a non-constant, null-homologous periodic orbit of the magnetic geodesic flow with negative action on the prescribed energy level \( \kappa \).
\begin{lem}\label{l: lemma to apply contact type criterion}
Let \( (M, g, \mathrm{d}\alpha) \) be a magnetic system with \( \dim M \geq 3 \), and let \( \kappa \in (0,\infty) \). Suppose there exists a non-constant, null-homologous periodic magnetic geodesic \( \gamma \) in \( (M, g, \mathrm{d}\alpha) \) of prescribed energy \( \kappa \), with non-positive action \( L + \kappa \).

Then the corresponding periodic orbit in the phase space, given by \( \Gamma(t) = (\gamma(t), \dot{\gamma}(t)) \), is null-homologous in \( \Sigma_{\kappa} \). Moreover, there exists a null-homologous, \( \varPhi^t_{g, \mathrm{d}\alpha} \)-invariant probability measure \( \mu \) supported on \( \Sigma_{\kappa} \) such that
\[
\int_{\Sigma_{\kappa}} \lambda(Z_E) \, \mathrm{d}\mu = 0.
\]
\end{lem}

\begin{proof}
First, observe that the set \( \{ \gamma(t) : t \in \mathbb{R} \} \) is precisely the image under the projection \( \pi \) of the closed flow line \( \Gamma \) of \( \varPhi^t_{g,\mathrm{d}\alpha} \), given by
\[
\{ \Gamma(t) := (\gamma(t), \dot{\gamma}(t)) : t \in \mathbb{R} \} \subseteq \Sigma_{\kappa},
\]
where \( \pi: \Sigma_{\kappa} \to M \) denotes the canonical projection. Hence, \( \pi_*([\Gamma]) = [\gamma] \), where \( [\Gamma] \in H_1(\Sigma_{\kappa}) \), \( [\gamma] \in H_1(M) \), and \( \pi_* \) is the induced map on homology.

Now, since \( \Sigma_{\kappa} \) is an \( \mathbb{S}^{\dim M - 1} \)-bundle over \( M \), we consider the long exact sequence in homotopy associated to the fibration
\[
\mathbb{S}^{\dim M - 1} \hookrightarrow \Sigma_{\kappa} \twoheadrightarrow M.
\]
Because \( \dim M - 1 \geq 2 \), we have \( \pi_1(\mathbb{S}^{\dim M - 1}) = 0 \), and thus the homomorphism
\[
\iota_*: \pi_1(\mathbb{S}^{\dim M-1}) \longrightarrow \pi_1(\Sigma_{\kappa})
\]
vanishes. From the exactness of the sequence, it follows that
\[
\pi_* : \pi_1(\Sigma_{\kappa}) \longrightarrow \pi_1(M)
\]
is an isomorphism. By the Hurewicz theorem, the induced map on homology
\[
\pi_*: H_1(\Sigma_\kappa) \longrightarrow H_1(M)
\]
is an isomorphism as well. Since \( [\gamma] \) is null-homologous in \( M \), it follows that \( [\Gamma] \) is trivial in \( H_1(\Sigma_{\kappa}) \), i.e., \( \Gamma \) is null-homologous in \( \Sigma_{\kappa} \).

Then the corresponding invariant probability measure \( \mu_{\Gamma} \) satisfies
\begin{equation}\label{eq: action of probability measure of Gamma is negative}
    S_{L+\kappa}(\mu_{\Gamma}) := \int_{\Sigma_{\kappa}} \big( L(x,v) + \kappa \big) \, \mathrm{d}\mu_{\Gamma} = \frac{S_{L + \kappa}(\gamma)}{T} \leq 0,
\end{equation}
where \( T \) denotes the period of \( \gamma \). Thus, \( [\mu_{\Gamma}] = 0 \) in the sense of \Cref{r: defi nullhomologous prob measure}, since \( \Gamma \) is null-homologous.

From here on, we follow the argument on pages 10–11 of~\cite{ConMacPat2004}. Let \( \mu_{\ell} \) denote the Liouville measure induced by the volume form associated with the canonical symplectic form \( \lambda \) on \( \Sigma_{\kappa} \), which also satisfies \( [\mu_{\ell}] = 0 \). Since \( \mu_{\ell} \) is invariant under the involution \( v \mapsto -v \), the transformation rule for integrals implies
\[
\int_{\Sigma_{\kappa}} \alpha_x(v)\, \mathrm{d}\mu_{\ell} = 0.
\]
Using the identity
\[
L(x,v) + \kappa = 2\kappa - \alpha_x(v) \quad \text{for all } (x,v) \in \Sigma_{\kappa},
\]
we obtain
\begin{equation} \label{eq:action_Liouville_positive}
S_{L+\kappa}(\mu_{\ell}) := \int_{\Sigma_{\kappa}} \big( L(x,v) + \kappa \big) \, \mathrm{d}\mu_{\ell} = 2\kappa > 0.
\end{equation}
Now, using \eqref{eq: action of probability measure of Gamma is negative}, \eqref{eq:action_Liouville_positive}, and the fact that \( [\mu_{\ell}] = [\mu_{\Gamma}] = 0 \) in \( H_1(\Sigma_{\kappa}, \mathbb{R}) \), there exists a constant \(0 \leq A \leq 1  \) such that the convex combination
\[
\nu := A \cdot \mu_{\Gamma} + (1-A) \, \cdot \mu_{\ell}
\]
defines a \( \varPhi^t_{g,\mathrm{d}\alpha} \)-invariant probability measure supported on \( \Sigma_{\kappa} \), and satisfies
\[
\int_{\Sigma_{\kappa}} \lambda(Z_E)\, \mathrm{d}\nu = 0.
\]
Here we use the well-known identity (see, for example, \cite[Lemma 3.4]{ConMacPat2004}):
\[
\lambda_{(x,v)}((Z_E)_{(x,v)}) = L(x,v) + \kappa \quad \text{for all } (x,v) \in \Sigma_{\kappa}.
\]
\end{proof}
\bibliographystyle{abbrv}
	\bibliography{ref}

\begin{thebibliography}{10}

\bibitem{AbbasCieliebakHofer2005}
C.~Abbas, K.~Cieliebak, and H.~Hofer.
\newblock {The Weinstein conjecture for planar contact structures in dimension three}.
\newblock {\em Commentarii Mathematici Helvetici}, 80:771--793, 2005.

\bibitem{Abbo13Lect}
A.~Abbondandolo.
\newblock Lectures on the free period {L}agrangian action functional.
\newblock {\em J. Fixed Point Theory Appl.}, 13(2):397--430, 2013.

\bibitem{AbbMacMazzPat17}
A.~Abbondandolo, L.~Macarini, M.~Mazzucchelli, and G.~P. Paternain.
\newblock Infinitely many periodic orbits of exact magnetic flows on surfaces for almost every subcritical energy level.
\newblock {\em J. Eur. Math. Soc. (JEMS)}, 19(2):551--579, 2017.

\bibitem{Abbondandolo2015}
A.~Abbondandolo, L.~Macarini, and G.~P. Paternain.
\newblock On the existence of three closed magnetic geodesics for subcritical energies.
\newblock {\em Commentarii Mathematici Helvetici}, 90:155--193, 2015.

\bibitem{AcuMoreno2022}
B.~Acu and A.~Moreno.
\newblock Planarity in higher-dimensional contact manifolds.
\newblock {\em International Mathematics Research Notices}, 2022(6):4222--4258, March 2022.

\bibitem{ABM}
P.~Albers, G.~Benedetti, and L.~Maier.
\newblock The {H}opf-{R}inow theorem and the {M}a\~n\'e critical value for magnetic geodesics on odd-dimensional spheres.
\newblock {\em Journal of Geometry and Physics}, 2025.

\bibitem{AlbersFuchsMerry2015}
P.~Albers, U.~Fuchs, and W.~J. Merry.
\newblock {Orderability and the Weinstein conjecture}.
\newblock {\em Compositio Mathematica}, 151(12):2251--2272, 2015.

\bibitem{albers2009weinstein}
P.~Albers and H.~W. Hofer.
\newblock {On the Weinstein conjecture in higher dimensions}.
\newblock {\em Commentarii Mathematici Helvetici}, 84(2):429--436, 2009.

\bibitem{ar61}
V.~I. Arnold.
\newblock Some remarks on flows of line elements and frames.
\newblock {\em Dokl. Akad. Nauk SSSR}, 138:255--257, 1961.

\bibitem{AssBenLust16}
L.~Asselle and G.~Benedetti.
\newblock The {L}usternik-{F}et theorem for autonomous {T}onelli {H}amiltonian systems on twisted cotangent bundles.
\newblock {\em J. Topol. Anal.}, 8(3):545--570, 2016.

\bibitem{Assenza24}
V.~Assenza.
\newblock Magnetic curvature and existence of a closed magnetic geodesic on low energy levels.
\newblock {\em International Mathematics Research Notices}, 2024(21):13586--13610, November 2024.

\bibitem{Ben14Contact}
G.~Benedetti.
\newblock The contact property for symplectic magnetic fields on {$S^2$}.
\newblock {\em Ergodic Theory Dynam. Systems}, 36(3):682--713, 2016.

\bibitem{BM24}
J.~Bimmermann and L.~Maier.
\newblock {Magnetic billiards and the Hofer--Zehnder capacity of disk tangent bundles of lens spaces}.
\newblock {\em Mathematische Annalen}, 2025.

\bibitem{Birkhoff1917}
G.~D. Birkhoff.
\newblock Dynamical systems with two degrees of freedom.
\newblock {\em Transactions of the American Mathematical Society}, 18(2):199--300, 1917.

\bibitem{Birkhoff1966}
G.~D. Birkhoff.
\newblock {\em Dynamical Systems}, volume~9 of {\em American Mathematical Society Colloquium Publications}.
\newblock American Mathematical Society, Providence, R.I., with an addendum by jürgen moser edition, 1966.

\bibitem{BormanMurphyEliashberg2015}
M.~S. Borman, Y.~Eliashberg, and E.~Murphy.
\newblock Existence and classification of overtwisted contact structures in all dimensions.
\newblock {\em Acta Mathematica}, 215(2):281--361, 2015.

\bibitem{Bott-Tu}
R.~Bott and L.~W. Tu.
\newblock {\em Differential Forms in Algebraic Topology}.
\newblock Springer New York, NY, 1982.

\bibitem{CFP10}
K.~Cieliebak, U.~Frauenfelder, and G.~P. Paternain.
\newblock Symplectic topology of {M}a\~n\'e's critical values.
\newblock {\em Geom. Topol.}, 14(3):1765--1870, 2010.

\bibitem{cineli2024closedorbitsdynamicallyconvex}
E.~Cineli, V.~L. Ginzburg, and B.~Z. G\"urel.
\newblock {Closed Orbits of Dynamically Convex Reeb Flows: Towards the HZ- and Multiplicity Conjectures}.
\newblock {\em arXiv preprint arXiv:2410.13093}, 2024.

\bibitem{Co06}
G.~Contreras.
\newblock The {P}alais-{S}male condition on contact type energy levels for convex {L}agrangian systems.
\newblock {\em Calc. Var. Partial Differential Equations}, 27(3):321--395, 2006.

\bibitem{CIPP}
G.~Contreras, R.~Iturriaga, G.~P. Paternain, and M.~Paternain.
\newblock Lagrangian graphs, minimizing measures and {M}a\~n\'e's critical values.
\newblock {\em Geom. Funct. Anal.}, 8(5):788--809, 1998.

\bibitem{ConMacPat2004}
G.~Contreras, L.~Macarini, and G.~P. Paternain.
\newblock Periodic orbits for exact magnetic flows on surfaces.
\newblock {\em Int. Math. Res. Not.}, (8):361--387, 2004.

\bibitem{contreras2024closedgeodesicsbettinumber}
G.~Contreras and M.~Mazzucchelli.
\newblock {Closed geodesics and the first Betti number}.
\newblock {\em arXiv preprint arXiv:2407.02995}, 2024.

\bibitem{cristofarogardiner2024proofhoferwysockizehndersinfinityconjecture}
D.~Cristofaro-Gardiner, U.~Hryniewicz, M.~Hutchings, and H.~Liu.
\newblock Proof of {H}ofer-{W}ysocki-{Z}ehnder's two or infinity conjecture.
\newblock {\em arXiv preprint arXiv:2310.07636}, 2024.

\bibitem{Eliashberg1989}
Y.~Eliashberg.
\newblock Classification of overtwisted contact structures on 3-manifolds.
\newblock {\em Inventiones mathematicae}, 98(3):623--637, 1989.

\bibitem{Fat97}
A.~Fathi.
\newblock Solutions kam faibles conjuguées et barrières de peierls.
\newblock {\em Comptes Rendus de l'Académie des Sciences. Série I. Mathématique}, 325:649--652, 1997.

\bibitem{FathiMaderna2007}
A.~Fathi and E.~Maderna.
\newblock Weak {KAM} theorem on non compact manifolds.
\newblock {\em NoDEA Nonlinear Differential Equations Appl.}, 14:1--27, 2007.

\bibitem{FHV89}
A.~Floer, H.~Hofer, and C.~Viterbo.
\newblock The {W}einstein conjecture in {$P\times {\bf C}^l$}.
\newblock {\em Math. Z.}, 203(3):469--482, 1990.

\bibitem{Frauenfelder_Schlenk}
U.~Frauenfelder and F.~Schlenk.
\newblock Hamiltonian dynamics on convex symplectic manifolds.
\newblock {\em Israel J. Math.}, 159:1--56, 2007.

\bibitem{Gg08}
H.~Geiges.
\newblock {\em An introduction to contact topology}, volume 109.
\newblock Cambridge University Press, 2008.

\bibitem{GeigesZehmisch2012}
H.~Geiges and K.~Zehmisch.
\newblock {Symplectic cobordisms and the strong Weinstein conjecture}.
\newblock {\em Mathematical Proceedings of the Cambridge Philosophical Society}, 153(2):261--279, 2012.

\bibitem{GeigesZehmisch2016}
H.~Geiges and K.~Zehmisch.
\newblock {The Weinstein Conjecture for Connected Sums}.
\newblock {\em International Mathematics Research Notices}, 2016(2):325--342, 2016.

\bibitem{Gin}
V.~L. Ginzburg.
\newblock {On closed trajectories of a charge in a magnetic field. {A}n application of symplectic geometry}.
\newblock In {\em {Contact and symplectic geometry ({C}ambridge, 1994)}}, volume~8 of {\em Publ. Newton Inst.}, pages 131--148. Cambridge Univ. Press, Cambridge, 1996.

\bibitem{Ginzburg2005}
V.~L. Ginzburg.
\newblock The weinstein conjecture and theorems of nearby and almost existence.
\newblock In J.~E. Marsden and T.~S. Ratiu, editors, {\em The Breadth of Symplectic and Poisson Geometry}, volume 232 of {\em Progress in Mathematics}, pages 139--172. Birkhäuser Boston, 2005.

\bibitem{Gr85}
M.~L. Gromov.
\newblock Pseudo holomorphic curves in symplectic manifolds.
\newblock {\em Inventiones Mathematicae}, 82:307--347, 1985.

\bibitem{hadamard1898}
J.~Hadamard.
\newblock {Les Surfaces {\`a} Courbures Oppos{\'e}es et Leurs Lignes G{\'e}od{\'e}siques}.
\newblock {\em Journal de Math{\'e}matiques Pures et Appliqu{\'e}es}, 4:27--73, 1898.

\bibitem{hadamard1899}
J.~Hadamard.
\newblock {Sur les G{\'e}od{\'e}siques d'une Surface {\`a} Courbure N{\'e}gative}.
\newblock {\em Comptes Rendus Hebdomadaires des S{\'e}ances de l'Acad{\'e}mie des Sciences}, 128:1020--1022, 1899.

\bibitem{Hofer1993}
H.~Hofer.
\newblock Pseudoholomorphic curves in symplectizations with applications to the {W}einstein conjecture in dimension three.
\newblock {\em Invent. Math.}, 114(3):515--563, 1993.

\bibitem{HV92}
H.~Hofer and C.~Viterbo.
\newblock The {W}einstein conjecture in the presence of holomorphic spheres.
\newblock {\em Communications on Pure and Applied Mathematics}, 45(5):583--622, 1992.

\bibitem{HoferWysockiZehnder2003}
H.~Hofer, K.~Wysocki, and E.~Zehnder.
\newblock {Finite energy foliations of tight three-spheres and Hamiltonian dynamics}.
\newblock {\em Annals of Mathematics}, 157(1):125--255, 2003.

\bibitem{HoferZehnder1987}
H.~Hofer and E.~Zehnder.
\newblock {Periodic solutions on hypersurfaces and a result by C. Viterbo}.
\newblock {\em Inventiones Mathematicae}, 90:1--9, 1987.

\bibitem{HoferZehnderBook}
H.~Hofer and E.~Zehnder.
\newblock {\em Symplectic invariants and {H}amiltonian dynamics}.
\newblock Birkh\"auser Advanced Texts: Basler Lehrb\"ucher. [Birkh\"auser Advanced Texts: Basel Textbooks]. Birkh\"auser Verlag, Basel, 1994.

\bibitem{HutchingsTaubes2009}
M.~Hutchings and C.~H. Taubes.
\newblock {The Weinstein conjecture for stable Hamiltonian structures}.
\newblock {\em Geometry \& Topology}, 13(2):901--941, 2009.

\bibitem{Luc10}
W.~L{\"u}ck.
\newblock Survey on aspherical manifolds.
\newblock In {\em European Congress of Mathematics}, pages 53--82. European Mathematical Society, Z{\"u}rich, 2010.

\bibitem{LystFetThm51}
L.~A. Lyusternik and A.~I. Fet.
\newblock Variational problems on closed manifolds.
\newblock {\em Doklady Akad. Nauk SSSR (N.S.)}, 81:17--18, 1951.

\bibitem{Man}
R.~Ma\~n\'e.
\newblock Lagrangian flows: the dynamics of globally minimizing orbits.
\newblock In {\em International {C}onference on {D}ynamical {S}ystems ({M}ontevideo, 1995)}, volume 362 of {\em Pitman Res. Notes Math. Ser.}, pages 120--131. Longman, Harlow, 1996.

\bibitem{MP10}
L.~Macarini and G.~Paternain.
\newblock On the stability of mañé critical hypersurfaces.
\newblock {\em Calculus of Variations and Partial Differential Equations}, 39:579--591, 2010.

\bibitem{maier2025geometrichydrodynamicsinfinitedimensional}
L.~Maier.
\newblock On geometric hydrodynamics and infinite dimensional magnetic systems.
\newblock {\em arXiv preprint arXiv:2506.00544}, 2025.

\bibitem{M24}
L.~Maier.
\newblock On {M}a{\~n}{\'e}'s critical value for the two-component {H}unter--{S}axton system and an infinite-dimensional magnetic {H}opf--{R}inow theorem.
\newblock {\em arXiv preprint arXiv:2503.12901}, 2025.

\bibitem{McDuff1987}
D.~McDuff.
\newblock Applications of convex integration to symplectic and contact geometry.
\newblock {\em Annales de l'Institut Fourier}, 37:107--133, 1987.

\bibitem{Me09}
W.~J. Merry.
\newblock Closed orbits of a charge in a weakly exact magnetic field.
\newblock {\em Pacific J. Math.}, 247(1):189--212, 2010.

\bibitem{poincare1890}
H.~Poincar{\'e}.
\newblock {Sur le Probl{\`e}me des Trois Corps et les {\'E}quations de la Dynamique}.
\newblock {\em Acta Mathematica}, 13:1--270, 1890.

\bibitem{poincare1905}
H.~Poincar{\'e}.
\newblock {Sur les Lignes G{\'e}od{\'e}siques des Surfaces Convexes}.
\newblock {\em Transactions of the American Mathematical Society}, 6(3):237--274, 1905.

\bibitem{Rabinowitz78}
P.~H. Rabinowitz.
\newblock Periodic solutions of {H}amiltonian systems.
\newblock {\em Comm. Pure Appl. Math.}, 31(2):157--184, 1978.

\bibitem{Rademacherclosedgeodesics2024}
H.-B. Rademacher.
\newblock Simple closed geodesics in dimensions {$\ge 3$}.
\newblock {\em J. Fixed Point Theory Appl.}, 26(1):Paper No. 5, 14, 2024.

\bibitem{Schlenk_Applications_Hofer_Geom}
F.~Schlenk.
\newblock Applications of {H}ofer's geometry to {H}amiltonian dynamics.
\newblock {\em Comment. Math. Helv.}, 81(1):105--121, 2006.

\bibitem{Sor}
A.~Sorrentino.
\newblock {\em Action-minimizing methods in {H}amiltonian dynamics}, volume~50 of {\em Mathematical Notes}.
\newblock Princeton University Press, Princeton, NJ, 2015.
\newblock An introduction to Aubry-Mather theory.

\bibitem{Str90}
M.~Struwe.
\newblock {Existence of periodic solutions of Hamiltonian systems on almost every energy surface}.
\newblock {\em Boletim da Sociedade Brasileira de Matemática}, 20:49--58, 1990.

\bibitem{Sullivan1976}
D.~Sullivan.
\newblock Cycles for the dynamical study of foliated manifolds and complex manifolds.
\newblock {\em Inventiones Mathematicae}, 36:225--255, 1976.

\bibitem{Tai83}
I.~A. Taimanov.
\newblock {The principle of throwing out cycles in Morse-Novikov theory}.
\newblock {\em Soviet Mathematics Doklady}, 27:43--46, 1983.

\bibitem{Tai92a}
I.~A. Taimanov.
\newblock Closed extremals on two-dimensional manifolds.
\newblock {\em Russian Mathematical Surveys}, 47:163--211, 1992.

\bibitem{Tai92b}
I.~A. Taimanov.
\newblock Closed non self-intersecting extremals of multivalued functionals.
\newblock {\em Siberian Mathematical Journal}, 33:686--692, 1992.

\bibitem{TaubesWeinstein07}
C.~H. Taubes.
\newblock The {S}eiberg-{W}itten equations and the {W}einstein conjecture.
\newblock {\em Geom. Topol.}, 11:2117--2202, 2007.

\bibitem{ViterboWeinstein87}
C.~Viterbo.
\newblock A proof of {W}einstein's conjecture in {${\bf R}^{2n}$}.
\newblock {\em Ann. Inst. H. Poincar\'e{} Anal. Non Lin\'eaire}, 4(4):337--356, 1987.

\bibitem{Weinstein78}
A.~Weinstein.
\newblock Periodic orbits for convex {H}amiltonian systems.
\newblock {\em Ann. of Math. (2)}, 108(3):507--518, 1978.

\bibitem{zoll1903}
O.~Zoll.
\newblock {{\"U}ber Fl{\"a}chen mit Scharen Geschlossener Geod{\"a}tischer Linien}.
\newblock {\em Mathematische Annalen}, 57:108--133, 1903.

\end{thebibliography}
	
\end{document}